\documentclass[10pt]{article}

\title{Stable isoperimetric ratios and the Hodge Laplacian of hyperbolic manifolds}
\author{Cameron Gates Rudd}
\usepackage[utf8]{inputenc}
\DeclareSymbolFont{bbold}{U}{bbold}{m}{n}
\DeclareSymbolFontAlphabet{\bb}{bbold}
\usepackage[T1]{fontenc}

\usepackage{comment}
\usepackage{microtype}
\usepackage{hyperref}
\usepackage{pinlabel}
\usepackage{float}
\usepackage{graphicx}

\usepackage[T1]{fontenc}
\usepackage[adobe-utopia]{mathdesign}

\makeatletter
\newcommand\RedeclareMathOperator{%
  \@ifstar{\def\rmo@s{m}\rmo@redeclare}{\def\rmo@s{o}\rmo@redeclare}%
}
\newcommand\rmo@redeclare[2]{%
  \begingroup \escapechar\m@ne\xdef\@gtempa{{\string#1}}\endgroup
  \expandafter\@ifundefined\@gtempa
     {\@latex@error{\noexpand#1undefined}\@ehc}%
     \relax
  \expandafter\rmo@declmathop\rmo@s{#1}{#2}}
\newcommand\rmo@declmathop[3]{%
  \DeclareRobustCommand{#2}{\qopname\newmcodes@#1{#3}}%
}

\@onlypreamble\RedeclareMathOperator
\makeatother
\usepackage{amsmath,amsthm,thmtools,cite,mathtools ,setspace,enumitem,tikz-cd}
\usepackage{setspace}
\usepackage[letterpaper,top=1in, bottom=1in, left=1in, right=1in]{geometry}
 
\theoremstyle{plain}
\newtheorem{thm}{Theorem}[section]
\newtheorem{bigthm}{Theorem}

\newtheorem{mainthm}{Theorem}

\newtheorem{lem}[thm]{Lemma}
\newtheorem{prop}[thm]{Proposition}
\newtheorem*{cor}{Corollary}
\declaretheorem[style=remark]{remark}
\DeclareMathOperator{\scl}{{\tt{scl}}}
\let\fill\relax
\DeclareMathOperator{\fill}{{\tt{fill}}}
\DeclareMathOperator{\vol}{vol}
\DeclareMathOperator{\inj}{inj}
\DeclareMathOperator{\diam}{diam}
\DeclareMathOperator{\len}{{\tt{len}}}
\DeclareMathOperator*{\esssup}{ess\,sup}

\DeclareMathOperator{\s}{{\text{star}}}

\DeclareMathOperator*{\supp}{supp}

\newcommand{\norm}[1]{\left\lVert#1\right\rVert}
\newcommand{\Q}{{\mathbb Q}}
\newcommand{\st}{\text{st}}

\newcommand{\Z}{{\mathbb Z}}
\newcommand{\e}{\varepsilon}
\newcommand{\G}{\Gamma}
\renewcommand{\H}{\mathbb H}

\newcommand{\R}{\mathbb R}

\renewcommand{\d}{ \partial}

\date{}
\begin{document}
\maketitle
\begin{abstract}
    We show that for a closed hyperbolic 3-manifold, the size of the first eigenvalue of the Hodge Laplacian acting on coexact 1-forms is comparable to an isoperimetric ratio relating geodesic length and stable commutator length with comparison constants that depend polynomially on the volume and on a lower bound on injectivity radius, refining estimates of Lipnowski and Stern. We use this estimate to show that there exist sequences of closed hyperbolic 3-manifolds with injectivity radius bounded below and volume going to infinity for which the 1-form Laplacian has spectral gap vanishing exponentially fast in the volume.
 \end{abstract}

 \section{Introduction}
The spectrum of the Hodge Laplacian is a fundamental and well studied geometric invariant of Riemannian manifolds. The Hodge theorem partitions the positive spectrum into exact and coexact eigenvalues. For differential forms of degree one, the exact eigenvalues contain exactly the data of the Laplacian acting on functions and are well understood. The coexact spectrum however is considerably more mysterious. Recently, the first coexact eigenvalue of the Hodge Laplacian of a closed hyperbolic 3-manifold has been related to other aspects of its geometry and topology. For the function Laplacian, the first eigenvalue is known to be comparable to the square of the isoperimetric Cheeger constant. In this paper, we derive a similar estimate for the first coexact eigenvalue, building on work of Lipnowski and Stern in \cite{LS} motivated by torsion growth in finite covers. We use this new estimate to construct the first examples of hyperbolic 3-manifolds with coexact eigenvalues exponentially small compared to volume.

Given a hyperbolic 3-manifold $M$, it is natural to try to extract information about $M$ from its finite covers. A deep and interesting conjecture of Bergeron-Venkatesh, L{\^{e}, and L\"uck (see \cite{BV}, \cite{thang}, and \cite{wolf}) asks in part whether the volume of $M$ can be found by studying the torsion in the homology of a family of finite covers of $M$. In studying this question, Bergeron, \c Seng\"un, and Venkatesh in \cite{BSV} relate the growth rate of the cardinality of the torsion in the first homology of a tower of covers of a closed \textit{arithmetic} hyperbolic 3-manifold to the spectrum of the Laplacian on 1-forms. In particular, they prove the following theorem, where the technical definitions are given below.

\begin{thm} (\cite{BSV}) If a sequence $M_n\to M_0$ of congruence covers of an arithmetic hyperbolic 3-manifold $M_0$ satisfies the few small eigenvalues, small Betti numbers, and simple cycles conditions, then the log torsion growth rate is proportional to the volume. In particular, one has \[\lim\limits_{n\to\infty} \frac{\log\left|H_1(M_n;\Z)_{\emph{torsion}}\right|}{\vol(M_n)} = \frac{1}{6\pi}.\]
\end{thm}

The conditions appearing in Theorem 1.1. are:

\begin{enumerate}
    \item (Few small eigenvalues) For all $\e>0$ there is a $c>0$ such that $$\limsup_{n\to\infty} \frac{1}{\vol(M_n)}\sum\limits_{0<\lambda<c}|\log\lambda| \leq \e,$$
where $\lambda$ runs over the eigenvalues of the 1-form Laplacian on $M_n$.

    \item  (Small Betti numbers) $b_1(M_n) = o\left(\frac{\vol(M_n)}{\log \vol(M_n)}\right)$.

    \item (Simple cycles) There exists a constant $C$ depending on $M_0$ such that $H_2(M_n;\Z)$ admits a basis of surfaces $[S_i]$ with bounded Thurston norm $||[S_i]||_{Th}\ll \vol(M)^C$.
\end{enumerate}

Bergeron, \c{S}eng\"un, and Venkatesh’s theorem provides a version of L\"uck approximation for the limiting $L^2$-torsion of sequences of manifolds that satisfy these conditions. However, no examples of such a sequence are known to exist. Encouraged by the Betti number approximation theorem of \cite{samurai}, one would hope the above conditions also imply the convergence of $L^2$-torsion for families of manifolds that Benjamini-Schramm converge to $\H^3$. Examples of Brock and Dunfield show that Benjamini-Schramm convergence itself is insufficient \cite{BDE}.

Previous work has focused on the simple cycles condition. In their paper, Bergeron, \c Seng\"un, and Venkatesh conjecture the simple cycles condition is satisfied for all arithmetic congruence covers, and ask in contrast if there are families of closed hyperbolic 3-manifolds with injectivity radius bounded below and volume going to infinity that do not satisfy it. This question was answered by Brock and Dunfield in \cite{BD}, where they construct a sequence of closed hyperbolic 3-manifolds $W_n$ with injectivity radius bounded below and volume going to infinity such that $H_2(W_n;\Z)\cong \Z$ for every $n$ and for which the Thurston norm of the generator grows exponentially in the volume.

In this paper, we continue the study of the few small eigenvalues condition initiated by Lipnowski and Stern in \cite{LS}. There, it is shown that for a family of covers of a triangulated closed hyperbolic $n$-manifold, the first positive eigenvalue of the Laplacian acting on 1-forms is comparable to a certain isoperimetric ratio, where the comparison constants depend on the volume of the cover, the geometry of the base, and the specific triangulation. In this paper, we prove that for closed hyperbolic 3-manifolds satisfying a uniform lower bound on injectivity radius, such a comparison can be done with universal constants and polynomial dependence on volume (Theorems A and B, with only Theorem A requiring the restriction to dimension 3). We then leverage the comparison in Theorem A to show that a specific family of closed hyperbolic 3-manifolds have first positive eigenvalue of the 1-form Laplacian vanishing exponentially fast in the volume (Theorem C). The manifolds in Theorem C are not covers of a fixed base, so the estimates of \cite{LS} do not apply.

Another motivation for this work comes from recent work of Lin and Lipnowski in \cite{LL}, where for closed rational homology 3-spheres they leverage a relationship between the first eigenvalue of the Hodge Laplacian acting on coexact 1-forms and irreducible solutions to the Seiberg-Witten equations to determine if certain spaces are $L$-spaces. In particular, if a closed rational homology 3-sphere $M$ is not an $L$-space, then the first eigenvalue $\lambda$ of the Hodge Laplacian acting on coexact 1-forms satisfies $\lambda\leq 2$. Using a version of the Selberg trace formula, they relate this to the complex length spectrum of $M$. Numerical methods can then be used to verify if $\lambda \leq 2$. While the constants in our eigenvalue estimates are rather opaque, so that the isoperimetric ratio we study is not, at least presently, capable of certifying that $\lambda\leq 2$, the relationship between the stable isoperimetric ratio and whether a space is an $L$-space remains tantalizing.

\subsection{Results}
The isoperimetric ratio we study relates the topological complexity of a surface with boundary to the geometric length of its boundary. One can view the extremal value of this ratio as an analogue of the two-dimensional Cheeger constant. The topological complexity measure is given by stable commutator length. The stable commutator length of a nullhomologous loop $\gamma$ in a manifold $M$, denoted $\scl(\gamma)$, is defined to be \[\scl(\gamma) = \inf\limits_{m\geq 1}\frac{{\tt{cl}}(\gamma ^{m})}{m},\] where ${\tt{cl}}(\gamma)$ is the word length of $\gamma$ in the commutator subgroup of $\pi_1M$ with generating set all commutators. Topologically, stable commutator length corresponds to the stable complexity of a surface bounding a nullhomologous loop. Denote the subgroup of rationally nullhomologous loops by $\Gamma_{\Q}’ = \ker\left(\pi_1M\to H_1(M;\Q)\right)$.
Then, for $\gamma\in \Gamma’_{\Q}$, one has $$\scl(\gamma) = \inf\left\{\frac{\chi_-(S)}{2m}~:~S \text{ with }\d S = \gamma^m,~S\text{ is connected}\right\},$$ where for connected surfaces $\chi_-(S) = \max\{0,-\chi(S)\}$. One can think of stable commutator length as a relative version of the Thurston norm. See the monograph \cite{Calegari} for a detailed exposition of $\scl$.
Stable commutator length is closely related to area, by Gauss-Bonnet, providing some justification for the following nomenclature. Define the stable isoperimetric constant of $M$ to be \[\rho(M) = \inf\limits_{\gamma\in \Gamma’_\Q\setminus\{1\}}\frac{|\gamma|}{\scl(\gamma)}.\]

One can also define the stable area of a rationally nullhomologous loop $\gamma\in \Gamma’_{\Q}$ to be the infimal normalized area of a surface bounding a power of $\gamma$: $$\text{sArea}(\gamma) = \inf\limits_{\d S = \gamma^n}\frac{\text{Area}(S)}{n}.$$ This leads to another notion of stabilized isoperimetric ratio using stable area in place of stable commutator length: $$\rho_{\text{Area}}(M) = \inf\limits_{\gamma\in\Gamma’_{\Q}}\frac{|\gamma|}{\text{sArea}(\gamma)}.$$
Stable area is related to stable commutator length by $$\text{sArea}(\gamma) \leq 4\pi\scl(\gamma).$$ Thus, $$\rho(M) \leq 4\pi \rho_{\text{Area}}(M).$$
Very little is known generally about how geodesic length relates to stable commutator length in hyperbolic manifolds, though estimates relating length, stable area, and stable commutator length for short curves in hyperbolic manifolds have been obtained by Calegari in \cite{length}.

Our first theorem relates the coexact spectral gap to stable isoperimetric ratios of arbitrary nullhomologous curves.

\begin{bigthm}\label{thm:A}
Let $M$ be a closed hyperbolic 3-manifold with injectivity radius bounded below by $\e>0$ and let $\lambda$ denote the first eigenvalue of the Hodge Laplacian acting on coexact 1-forms. Then there is a constant $A = A(\e)$ such that for any nontrivial element $\gamma \in \Gamma’_\Q$, one has $$\sqrt{\lambda} \leq A \vol(M) \frac{|\gamma|}{ \scl( \gamma)},$$ where $|\gamma|$ denotes the geodesic length of $\gamma$.
\end{bigthm}

Theorem A has the following obvious corollary:

\begin{cor}
  Let $M$ be a hyperbolic $3$-manifold with injectivity radius bounded below by $\e>0$ and let $\lambda$ be the first eigenvalue of the Hodge Laplacian acting on coexact 1-forms. Then for the constant $A = A(\e)$ from Theorem A, $$ \sqrt{\lambda} \leq A \vol(M)\rho(M).$$
\end{cor}

The analogue of Theorem A in \cite{LS} studies a cochain version of the Hodge Laplacian introduced by Dodziuk in \cite{Dodziuk} for triangulated manifolds. This chochain Laplacian is called the Whitney Laplacian, and is induced by the Hodge Laplacian by embedding cochains into the $L^2$-de Rham complex.

\begin{thm} (\cite{LS} Theorem 1.4) Let $M_0$ be a closed hyperbolic $n$-manifold. Let $K_0$ be a sufficiently fine triangulation. Let $M$ be a finite cover of $M_0$. Let $\lambda_W(M)_{d^*}$ be the first coexact eigenvalue for the Whitney cochain Laplacian associated to the pullback of the triangulation $K_0$ to $M$. Then if some multiple of $\gamma\in\pi_1(M)$ bounds a surface, then $$\left(\frac{\scl(\gamma)}{|\gamma|}\right)^2 \leq W_{M_0}\frac{\vol(M)}{\lambda_W(M)_{d^*}},$$
for a constant $W_{M_0}$ depending on the triangulation $K_0$.
\end{thm}

Under a well behaved sequence of subdivisions, Dodziuk and Patodi in \cite{dP} showed the spectrum of the Whitney Laplacian converges to the spectrum of the Hodge Laplacian. For a fixed but very fine triangulation, the eigenvalue comparison is somewhat delicate. In this setting, Lipnowski and Stern relate the Whitney Laplacian’s first coexact eigenvalue to the Hodge Laplacian’s first coexact eigenvalue in the following way:

\begin{thm} (\cite{LS} Theorem 1.5)	 Let $M_0$ be a closed hyperbolic $n$-manifold. Let $K_0$ be a sufficiently fine triangulation of $M_0$. Let $M$ be a finite cover of $M_0.$ Let $\lambda_W(M)_{d^*}$ be the first coexact eigenvalue for the Whitney cochain Laplacian associated to the pullback of the triangulation $K_0$ to $M$. Then, $$\frac{1}{\lambda_W(M)_{d^*}}\leq \max\left\{\frac{4G_{M_0}^2\vol(M)}{\lambda_{d^*}(M)}, G_{M_0}^2C_{M_0}^2\vol(M)\right\}.$$
The constants $C_{M_0}$ and $G_{M_0}$ depend only on $K_0$.
\end{thm}

In the course of proving Theorem A, we too require a comparison of this sort. By using a smoothed version of the Whitney embedding of cochains into the de Rham complex and triangulations with uniformly controlled geometry (these are called deeply embedded triangulations and are introduced in Section 2), we prove the following Whitney-Hodge eigenvalue comparison.

\begin{thm} Let $M$ be a closed hyperbolic 3-manifold with $\inj(M)>\e$. Let $\lambda$ denote the first coexact eigenvalue for the Hodge Laplacian acting on $1$-forms and let $\lambda_W$ denote the first coexact eigenvalue for the smoothened Whitney Laplacian on 1-cochains associated to a deeply embedded triangulation. There is a constant $G = G(\e)$ such that $$\lambda \leq G \vol(M)\lambda_W.$$
\end{thm}

The smoothened version of the Whitney map needed for the above proposition is obtained by replacing the barycentric coordinates associated to a triangulation with certain smooth partitions of unity indexed by the vertices of a triangulation (an idea of Dodziuk’s \cite{dodzuik2}), which we call barycentric partitions of unity. One can then show that the forms built from these pieces have well behaved Hodge decompositions. This analysis leads to the above eigenvalue comparison for the smoothened Whitney Laplacian constructed from a barycentric partition of unity.

Our second theorem uses the isoperimetric constant $\rho(M)$ to provide a lower bound on the first coexact eigenvalue of the 1-form Laplacian.

\begin{bigthm} Let $M$ be a closed hyperbolic $n$-manifold with $\inj(M) > \e$. Let $\lambda$ be the first positive eigenvalue for the Hodge Laplacian acting on coexact 1-forms and let $H > \lambda$. Then there is a constant $P(H,\e,n)>0$ such that $$ \frac{P\rho(M)}{\vol(M)^{7/2+1/n}}\leq\sqrt{\lambda}.$$
\end{bigthm}

Theorem B corresponds to Theorem 1.5 below from \cite{LS}, which uses stable area in place of stable commutator length. Note that Theorem B above remains true if one replaces $\rho(M)$ with $\rho_{\text{Area}}(M)$.

\begin{thm}(\cite{LS} Theorem 1.3 )  Let $M_0$ be a closed hyperbolic $n$-manifold and let $M$ be a finite cover of $M_0.$ Then there are constants $A_0$ and $C$, where $A_0$ is depends only on $M_0$ and $C$ is a constant that is uniformly bounded when the injectivity radius of $M$ is bounded below and $\lambda_1^1(M)$ is bounded above, for which $$\frac{1}{\lambda_1^1(M)_{d^*}}\leq A_0C^2\vol(M)^{3/2}\diam(M)^2 \left(1+\rho_{\emph{Area}}(M)^{-1}\right).$$
\end{thm}

Our approach to proving Theorems A and B is grounded in the following dual characterizations of $\scl$. By Bavard duality, stable commutator length is related to the defect norm for quasimorphisms and the Gersten filling norm for singular 1-chains. The Gersten filling norm for a nullhomologous loop $\gamma$ is given by the infimal $\ell^1$-norm of a singular 2-chain whose boundary is a fundamental cycle for $\gamma^m$, normalized by $m$. A quasimorphism for a group $\Gamma$ is a map from $\Gamma\to \R$ that is nearly a homomorphism in the sense that its coboundary is a bounded map on $\Gamma^2$. The defect of a quasimorphism is the sup norm of its coboundary. Bavard duality says $$\scl(\gamma) = 4\fill(\gamma) = \frac{1}{2}\sup\limits_q \frac{q(\gamma)}{D(q)},$$ where the supremum is over all quasimorphisms $q$ and $D(q)$ is the defect of $q$.

One can therefore use the characterization of $\scl$ as a filling norm when bounding it from above and, similarly, the quasimorphism point of view when bounding it from below. For Theorem A, we relate the filling norm to the spectrum of the Hodge Laplacian via the Whitney Laplacian and Poincar\'e duality (which forces us to restrict to dimension 3). For Theorem B, we use de Rham quasimorphisms, which are given by integrating coclosed forms over geodesics. Studying the de Rham quasimorphism of a coexact eigenform gives the connection to the spectrum of the Hodge Laplacian.

Our methods primarily differ from \cite{LS} in that instead of studying covers of a fixed manifold with a specific triangulation, we use that closed hyperbolic manifolds with injectivity radius greater than some $\e>0$ can all be triangulated so that the simplices come from a compact collection, and instead of using an $L^2$ discretization of the eigenvalue problem, we use a smooth discretization. The local structure of these triangulations can then be compared in a uniform way, thereby allowing us to relate various combinatorial and geometric norms. By working in the smooth setting instead of the $L^2$ setting, we are able to make use of geometric estimates that require higher regularity. This leads to the more direct Whitney-Hodge eigenvalue comparison of Proposition 1.4.

As an application of the spectral gap estimate of Theorem A, we modify the construction in \cite{BD} to construct a family of closed hyperbolic manifolds for which we have control over the stable isoperimetric constant and which have injectivity radius uniformly bounded below.

\begin{bigthm}\label{thm:C} There is a family $\{W_n\}$ of closed hyperbolic 3-manifolds with injectivity radius bounded below by some $\e>0$ and volume growing linearly in $n$ such that the 1-form Laplacian spectral gap vanishes exponentially fast in relation to volume: $$\sqrt{\lambda(W_n)}\leq B \vol(W_n)e^{-r\vol(W_n)}$$ where $r$ and $B$ are positive constants and $\lambda(W_n)$ is the first positive eigenvalue of the 1-form Laplacian on $W_n$.
\end{bigthm}

The manifolds $W_n$ in Theorem $C$ are obtained by taking a hyperbolic 3-manifold with totally geodesic boundary and gluing it to itself using a particular psuedo-Anosov with several useful properties. By \cite{BMNS}, this family has geometry that up to bounded error can be understood in terms of a simple model family and therefore has the desired injectivity radius lower bound and linear volume growth.

Using this model family, we show that one can find curves with uniformly bounded length whose stable commutator length grows exponentially in the volume. This task is rather delicate, as it is quite difficult to ensure that the \emph{stable} commutator length of a given loop is large--- one typically does not know if passing to a power of the loop causes a drastic simplification in commutator length. To overcome this, we require an algebraic condition on the gluing map to control which surfaces bound the loops we study. We then use the following homological isoperimetric estimate to control stable commutator length.

 \begin{thm} Let $M$ be a compact oriented hyperbolic 3-manifold with totally geodesic boundary $\d M = S$. Let $\gamma$ be a geodesic multicurve in $S$ that is rationally nullhomologous in $M$. Let $||\cdot||_{s,S}$ be the stable norm on $H_1(S)$ induced by the Riemannian metric on $S.$ Then there is a constant $D> 0$ depending only on $M$ such that $$||[\gamma]||_{s,S}\leq D\scl_M(\gamma).$$ \end{thm}

 Theorem A then implies the first positive eigenvalue vanishes exponentially fast. We stress that while the manifolds $W_n$ are constructed from the same basic building blocks, they are not otherwise related to one another. In particular, they are not finite covers of some fixed base manifold, so the estimates of \cite{LS} do not apply; it is therefore vital that the constants appearing in Theorem A depend only on an injectivity radius lower bound.

\subsection{Brief Outline} In \hyperref[sec:2]{Section 2}, we show existence and study the local properties of the triangulations we use throughout the paper. We also introduce the smooth Whitney cochain map used to define the approximation of the Laplacian. In \hyperref[sec:3]{Section 3}, we relate various chain and cochain norms and compare these to various geometric norms.  In \hyperref[sec:4]{Section 4}, we compare the eigenvalues of the approximation of the Laplacian to the genuine eigenvalues of the Laplacian, then use this comparison and the estimates from \hyperref[sec:3]{Section 3} to  prove Theorem A. In \hyperref[sec:5]{Section 5}, we prove Theorem B. Finally, in \hyperref[sec:6]{Section 6}, we compare the homological length and stable commutator length of certain curves and construct the example of Theorem C. \hyperref[sec:6]{Section 6} makes use of Theorem A, but is otherwise independent from the rest of the paper.

\begin{remark}
Throughout the paper, numerous constants are used. Constants defined inside proofs have no meaning outside the local setting of the proof. The letter $C$ is repeatedly reused in Sections 2 and 3 to denote a constant coming from a Sobolev type estimate and can at any time be taken to be the maximum among all constants denoted by $C$. Such constants depend only on injectivity radius and local choices of things like bump functions, unless specifically noted.
\end{remark}

\section*{Acknowledgments}  Countless thanks are due to my advisor Nathan Dunfield for suggesting the questions considered in this paper and for invaluable support. Thanks to Michael Lipnowski and Joel Hass for answering some questions. Thanks to the anonymous referee for helpful comments that improved this paper. This work was partially supported by US NSF grant DMS-1811156.

\section{Triangulations and Whitney forms}
\label{sec:2}

The purpose of this section is to outline the basic properties of the triangulations we use in this paper and how the Whitney map relates these triangulations to the de Rham complex.

\subsection{Deeply embedded triangulations}

In this section we study certain triangulations, called deeply embedded triangulations, of hyperbolic manifolds with injectivity radius bounded below that enjoy useful combinatorial and geometric properties that will facilitate the estimates in sections 3 through 5. While we focus on the hyperbolic setting, we give an account motivated by potential generalizations to the variable negative curvature setting.

The triangulations we use are generally obtained via Delaunay complexes associated to collections of points. To obtain a Delaunay complex in a Riemannian manifold $M$, take a finite collection of points $P\subset M$ and consider the Voronoi celluation consisting of cells $$V_p = \{x\in M~:~ d(x,p) \leq d(x,q)~\text{for all } p \neq q \in P\}$$ for $p\in P.$ Dual to the Voronoi celluation is the Delaunay complex. The cells of the Delauney complex are the convex hulls of tuples of points whose corresponding Voronoi cells have nonempty intersections. In \cite{Bois}, it is shown that if the collection of points $P$ satisfies certain density and separation conditions, then there is a quantifiably small perturbation of the point set $P$ whose Delaunay complex is a triangulation. The simplices of this triangulation are geodesic. The precise conditions are as follows.

Let $M$ be a closed Riemannian manifold with distance function $d$. Given a pair $1\geq\mu>0,~\e>0$, a $(\mu,\e)$-net is a collection $P$ of points in $M$ for which the following hold:

\begin{enumerate}
    \item ($P$ is $\e$-dense) For all $x\in M$, there is a $p\in P$ such that $d(x,p)<\e$.

    \item ($P$ is $\mu$-separated) All distinct $p,q\in P$ satisfy $d(p,q) \geq \mu\e$.

\end{enumerate}

Theorem 3 of \cite{Bois} says that if $\mu$ and $\e$ satisfy several inequalities relating to the injectivity radius and sectional curvatures, a $(\mu,\e)$-net can be perturbed to $(\mu’,\e’)$-net such that the resulting Delaunay complex is indeed a triangulation. This theorem, specialized to closed hyperbolic $n$-manifolds, becomes:
\begin{thm}

\cite{Bois} Let $M$ be a closed hyperbolic $n$-manifold and $P$ a $(\mu,\e)$-net such that \[\e \leq \min\left\{\frac{\inj(M)}{4},~\Psi(\mu)\right\},\] where $\inj(M)$ is the injectivity radius of $M$ and $\Psi$ is a function of the net parameter $\mu$. The function $\Psi$ is described in \cite{Bois} and is independent of the manifold $M$. Then there is a point set $P’$ that is a $(\mu’,\e’)$-net with resulting Delaunay complex a triangulation. Moreover, $\mu’$ and $\e’$ satisfy the following (see equation (2) in \cite{Bois}):
: \begin{align*}
     \e &\leq \e’ \leq \frac{5}{4}\e, \\
    \frac{2}{5}\mu &\leq \mu’ \leq \mu.
\end{align*}
\end{thm}

The separation and density conditions of a $(\mu’,\e’)$-net ensure that the resulting Delaunay triangulation has edge lengths lying in the closed interval $[\frac{2}{5}\mu \e ,2\e]$. We now specialize to the case $\mu=1$. Note that when $n=3$, $\Psi(1) = 2\times 3^{121.5}\times 5^{-81}$, which is roughly 45.15.

Set $\epsilon_0 = \min\{\e/10,\Psi(1)\}$ and define $\mathcal{G}_{\e}$ to be the space of hyperbolic $n$-simplices with edge lengths in the interval $[\frac{2\epsilon_0}{5} ,2\epsilon_0]$. Because the space of all hyperbolic $n$-simplices is parametrized by edge lengths, the space $\mathcal{G}_{\e}$ is compact. Proposition \ref{prop: existence} below establishes that every closed hyperbolic manifold $M$ with injectivity radius bounded below by $\e$ admits a triangulation whose simplices are isometric to those in $\mathcal{G}_\e$.

In a triangulation $K$, the combinatorial 1-neighborhood of a simplex is the union of the stars of the vertices of that simplex. A triangulation $K$ of a hyperbolic manifold $M$ is an \textbf{$\e$-deeply embedded triangulation} if:

\begin{enumerate}
    \item Every simplex is geodesic and contained in $\mathcal G_\e$,
    \item The combinatorial 1-neighborhood of every simplex lifts isometrically to $\H^n$.
\end{enumerate}

Throughout, when referring to deeply embedded triangulations, we generally suppress reference to some fixed $\e$.

\begin{prop} \label{prop: existence}
Let $M$ be a closed hyperbolic $n$-manifold and let $0<\e < \inj(M)$. Then there is a deeply embedded triangulation $K$ of $M$. That is, there is a geodesic triangulation $K$ of $M$ whose simplices come from $\mathcal G_{\e}$ such that the combinatorial 1-neighborhood of every simplex isometrically embeds in $\H^n$.
\end{prop}

\begin{proof} Set $\epsilon_0 = \min\{\e/10,\Psi(1)\}$.
 Take a maximal collection of points $P\subset M$ such that the balls $B_{\epsilon_0/2}(p)$ for $p\in P$ are all disjoint. By maximality, the $B_{\epsilon_0}(p)$ balls then cover $M$.
 Since the $B_{\epsilon_0/2}(p)$ balls are all disjoint, we have that for all $p,q\in P$, $d(p,q) \geq \epsilon_0 = \mu {\epsilon_0} $, so $P$ is $\mu$-separated. Since the $B_{\epsilon_0}(p)$ balls cover, every point $x\in M$ is ${\epsilon_0} $-close to some point $p\in P$, so $P$ is ${\epsilon} $-dense.
 The collection $P$ therefore is a $(\mu, {\epsilon_0})$-net, with $\mu = 1$ and ${\epsilon_0} $ satisfying the hypotheses of Theorem 2.1. Consequently, there is a perturbation of $P$ that is a $(\mu’,{\epsilon_0}’)$-net whose Delaunay complex is a triangulation. Since, as remarked above, the edge lengths of simplices in this Delaunay triangulation lies in the interval $[\frac{2}{5} {\epsilon_0} ,2 {\epsilon_0}]$, the simplices come from $\mathcal G_{\e} $.
 The edge length bound along with the fact ${\epsilon_0} < \inj(M)/10$ ensures the diameter of any vertex star (which will be less than 3 times the length of the longest edge of a simplex) will be less than $2 {\epsilon_0}$. Thus, the star of every vertex embeds isometrically in $\H^n$ via the local inverse of the exponential map. \end{proof}

We will also use cell complexes that are dual to deeply embedded triangulations. Every simplicial triangulation $K$ of a closed Riemannian manifold admits a dual celluation $K^*$ comprised of cells $\sigma^*$ dual to the simplices $\sigma$ of the triangulation $K$ in the following sense (for a reference, see \cite{bredon} chapter VI.6): Take the first barycentric subdivision $\tau(K)$ of $K$, then the $n$-cells of the dual celluation $K^*$ are the closed stars of the vertices of the original triangulation $K$ in the barycentric subdivision. This celluation is naturally triangulated by the the barycentric subdivision triangulation $\tau(K)$. Like $K$, the dual celluation can be uniformly controlled.

The controlled geometry of deeply embedded triangulations and their dual celluations primarily manifests in the compactness of $\mathcal G_\e$ and the following local bound.

\begin{prop} \label{prop: star bound}Let $M$ be a closed hyperbolic $n$-manifold of injectivity radius $\inj(M)>\e$ with a deeply embedded triangulation $K$. Then there is a positive constant $N = N(\e)$ such that the number of $k$-simplices in the star of a $j$-simplex is less than $N$.
\end{prop}

\begin{proof}

The edge length bounds provide a lower bound on the angle between two edges meeting at a vertex that span a face via the hyperbolic law of cosines. This implies that the number of $n$-simplices meeting at a vertex $v$ is bounded uniformly, since for any ball around the vertex, there is a uniform lower bound on the volume of the intersection of an $n$-simplex containing the vertex $v$ and the ball. It follows that there is an $N$ such that the number of $k$-simplices in the star of a vertex is less than $N$ for $k = 0,\dots, n$.\end{proof}

Proposition \ref{prop: star bound} and the compactness of the space $\mathcal G_\e$ of simplices together imply that the geometry and combinatorics of the 1-neighborhood of a simplex in a deeply embedded triangulation is uniformly controlled.
In particular, let $K$ be a deeply embedded triangulation of $M$. By Proposition \ref{prop: star bound}, any vertex in $K$ is contained in at most $N$ simplices. Therefore, there are finitely many possible finite simplicial complexes that appear as the combinatorial 1-neighborhood of a simplex in a deeply embedded triangulation. Let $\mathcal C$ be the finite set of such possible complexes.

For any complex $a \in \mathcal C$, say with $|a|$ many $n$-simplices, a hyperbolic structure on $a$ is given by identifying each $n$-simplex in $a$ with a model simplex in $\mathcal G_{\e}$ so that the face gluing maps are isometries. The possible geometric structures on $a$ are parametrized by a subspace $\mathcal S_\e(a)$ of $\mathcal G_{\e}^{|a|}$. Because the gluing conditions are closed, and because $\mathcal G_\e$ is compact, the space $\mathcal S_\e(a)$ is compact.
By taking the disjoint union over the finite list of possible complexes $a\in\mathcal C$, there is a compact space $\mathcal S_{\e}$ that parametrizes the geometry of the combinatorial 1-neighborhood of a simplex in a deeply embedded triangulation.

We now turn to relating closed geodesics in $M$ to cellular paths in $K$.

\begin{prop} \label{prop: loop multiplicity}
Let $M$ be a closed hyperbolic $n$-manifold with injectivity radius $\inj(M)>\e$ with a deeply embedded triangulation $K$. Let $\gamma$ be a closed geodesic curve in $M$. Then there is a constant $J = J(\e)$ such that the number $v$ of cells in the dual cell complex $K^*$ that $\gamma$ intersects (counted with multiplicity) satisfies \[v\leq J|\gamma|.\] \end{prop}

\begin{proof}

Suppose $\gamma$ moves from an $n$-cell $\sigma$ to an $n$-cell $\sigma’$, intersecting the ($n-1)$-skeleton of $K^*$ at a point $p\in \sigma\cap \sigma’$. Consider the closed radius $\e$-ball at the point $p$, $V = \bar B_{\e}(p).$ Let $x$ be the point at which $\gamma$ enters $V$ and let $y$ be the point at which it exits $V$.
Then the geodesic subarc of $\gamma$ running from $x$ to $y$ has length $2\e$. Since the triangulation $K$ has simplices from $\mathcal G_{\e}$, the restrained combinatorics of the dual celluation ensures that the ball $V$ intersects a universally bounded number of dual cells. Let $R(\e)$ denote this bound.

Consider the sequence $x_n,y_n$ of points such that $x_1 = x$ and $y_1 = y$ from above for the first simplex crossing, and $x_n$ is obtained by taking the simplex crossing that happens after $y_n$. Then each pair $x_n,y_n$ corresponds to a geodesic sub arc of $\gamma$ that intersects at most $R(\e)$ simplices. Thus, $\nu \leq (\frac{|\gamma|}{2\e} +1)R(\e)$.
It therefore follows that $\frac{v}{R(\e)} 2\e \leq |\gamma| + 2\e.$ Since $\e \leq |\gamma|$, we have $|\gamma|+2\e\leq 3|\gamma|$, and the stated linear bound follows with $J = \frac{3R(\e)}{2\e}$. \end{proof}

The next result compares the lengths of closed geodesics in $M$ to approximating paths in the 1-skeleton of dual celluation $K^*$. To measure the complexity of paths in $K^*$, let $||\cdot||_G$ be the $\ell^1$-norm on chains and $\len(\cdot)$ the word length of the cellular path. For a cellular path $c$, let $||c||_G$ be the $\ell^1$-norm of the corresponding chain.

\begin{prop} \label{prop: length comparison} There is a constant $L = L(\e)>0$ such that for any closed geodesic curve $\gamma$ in $M$, there is a cellular path $c$ in $K^*$ homotopic to $\gamma$ such that $||c||_G\leq \len(c) \leq L|\gamma|$.
\end{prop}

\begin{proof}

Fix a base point and orientation for $\gamma$ such that the base point lies on a face of a top dimensional cell. The curve $\gamma$ can be replaced by a homotopic curve whose length is bounded by a constant times the geodesic length of $\gamma$ and which intersects the boundary of every simplex at vertices. This follows from Proposition \ref{prop: loop multiplicity} which gives that there is a bound on the number of simplices $\gamma$ intersects (counting these intersections with multiplicity) that depends linearly on the length of $\gamma$ and the fact the simplices of the triangulation have bounded diameter. Using the orientation and basepoint, we obtain a sequence of vertices with line segments between them that lie entirely in a cell. We can further modify $\gamma$ by replacing these curve segments with curves that lie in the 1-skeleton by traversing the 1-simplex joining the two boundary vertices. Since the edges in the celluation $K^*$ have bounded length, this again adds bounded length to the curve. Let $c$ denote the cellular path we have constructed. By the previous considerations, there is a constant $L$ depending only on $\e$ giving the comparison $\len(c)\leq L|\gamma|$. The inequality $||c||_G\leq \len(c)$ is trivial.

\end{proof}

The geometry of geodesic simplices in hyperbolic manifolds can also be understood using barycentric coordinates via Thurston’s straightening map (see \cite{thurstonbook} page 124). Identify $\H^n$ with the upper sheet of the hyperboloid in Minkowski space $\R^{n,1}$ with quadratic form $Q = x_0^2 + x_1^2 + \cdots + x_{n-1}^2 - x_n^2$ and consider a singular simplex $\sigma:\Delta\to \H^n$.
Let $b_0, \dots, b_n:\Delta\to [0,1]$ be the barycentric coordinates on the standard Euclidean simplex $\Delta$ with vertices $e_0,\dots, e_n$. Then for $v\in \Delta$, write $v = \sum b_i(v)e_i$. The straightening of $\sigma$ is the singular simplex in $\H^n$ given by centrally projecting the affine simplex $\sum b_i(v)\sigma(v_i)$ from the origin to the upper sheet of the hyperboloid. This process endows each geodesic simplex with a natural barycentric coordinate.

If $\pi:\H^n\to M$ is the projection map and $\sigma$ is a singular simplex in $M$, let $\st(\sigma):\Delta\to M$ be the composition of the straightening of $\sigma$ applied to some lift of $\sigma$ and the projection map.
This is well-defined and independent of the lift because the isometry group of $\H^n$ acts linearly on $\R^{n,1}$ preserving the quadratic form $Q$.

For a geodesic simplex $\sigma$ in $\H^n$, let $V_{\sigma}:\sigma \to \Delta$ be the map from $\sigma$ to the standard simplex given by the barycentric coordinates induced by straightening. Geodesic simplices $\sigma,\sigma'$ can be compared using the composition $V_{\sigma’}^{-1}\circ V_{\sigma}$.
Using the straightening construction and the barycentric coordinates, one sees that the maps $V_{\sigma}$ depend continuously on $\sigma$ in the sense that if $\sigma$ is a straight simplex in $\H^n$ and $\sigma’$ is obtained by perturbing the vertices of $\sigma$, then the composition map $V_{\sigma’}^{-1}\circ V_{\sigma}$ is almost an isometry, where the failure to be an isometry is controlled by the size of the vertex perturbation.

\begin{prop}\label{prop: simplices biLipschitz} Let $\sigma$ and $\sigma’$ be geodesic simplices from $\mathcal G_\e$ embedded in $\H^n$. Then the map $V_{\sigma’}^{-1}\circ V_{\sigma}$ is $\kappa$-biLipschitz for some $\kappa = \kappa(\e)>0$ that does not depend on $\sigma$ and $\sigma’$.
\end{prop}
\begin{proof}
The biLipschitz constant for the comparison map between any given simplex and the Euclidean simplex depends continuously on the simplex. The result then follows from the compactness of $\mathcal G_\e$.
\end{proof}

We also have uniform control over the geometry of the combinatorial 1-neighborhood of a simplex.

\begin{prop}\label{prop: stars biLipschitz} There is a constant $\mathcal L = \mathcal L(\e)$ such that  if $s$ and $s’$ are two complexes of the same combinatorial type in $\mathcal S_\e$, then $s$ and $s’$ are $\mathcal L$-biLipschitz equivalent.
\end{prop}
\begin{proof}
    Define maps from the combinatorial 1-neighborhood $s$ of a simplex $\sigma$ to an abstract Euclidean model by gluing the maps $V_{\sigma'}$ together according to the combinatorics of $s$. Since the gluing maps are isometries, this is well defined. Since the map restricted to each simplex $\sigma'$ is uniformly biLipschitz equivalent to the model simplex, and since there are a uniformly bounded number of simplices in $s$, it follows that $s$ is uniformly biLipschitz equivalent to the Euclidean model.
\end{proof}

Lastly, we note that the lower bound on volumes of simplices in deeply embedded triangulations implies the following.

\begin{prop} \label{prop: volume}There is a constant $T=T(\e)$ such that if $M$ is a closed hyperbolic $n$-manifold with $\inj(M)>\e$ and $K$ is a deeply embedded triangulation of $M$, then the number $V_K$ of simplices in $K$ satisfies $$V_K\leq T\vol(M).$$
\end{prop}

\subsection{Whitney forms}

The combinatorial geometry of the triangulation $K$ is related to the Riemannian geometry of $M$ by way of the Whitney form map $W:C^{\bullet}(K)\to L^2\Omega^{\bullet}(M)$ relating the cochain complex $C^{\bullet}(K)$, with $\R$ coefficients, to the $L^2$-de Rham complex. It will be useful to view $C^{\bullet}(K)$ as a subcomplex of the singular cochain complex. With this in mind, we often identify singular simplices with their images.

The Whitney map is readily defined using the basis for $C^{\bullet}(K)$ dual to the basis of simplices. This basis consists of the cochains $\delta_{\sigma}$ that take the value 1 on the simplex $\sigma$ and zero on all other simplices.
The Whitney form $W(\delta_{\sigma})$ associated to the cochain $\delta_{\sigma}$ dual to an  oriented simplex $\sigma = [v_0,\dots,v_q]$ is given by
$$W(\delta_{\sigma}) = q! \sum_{k=0}^q(-1)^kb_{k}db_0\wedge \cdots \wedge db_{k-1}\wedge db_{k+1}\wedge \cdots \wedge d b_q,$$
where $b_k:M\to [0,1]$ is the barycentric coordinate associated to the vertex $v_k$. See \cite{Dodziuk} for more details.

An $L^2$-form in the image of $W$  is called a Whitney form. The support of a Whitney form $W(\delta_{\sigma})$ is contained in the closed star of the simplex $\sigma$. The barycentric coordinates used to define the Whitney forms are not smooth, however, they are smooth in the compliment of the $(n-1)$-skeleton of $K$. One can define the exterior derivative of a Whitney form in a weak sense, which yields a differential that is well defined as an $L^2$-form. With this exterior derivative, the Whitney map becomes a chain map. For any cochain $f$ and simplex $\sigma$, the restriction of the Whitney cochain $\omega= W(f)$ to $\sigma$, denoted $\omega|_{\sigma}$, can be uniquely extended to a smooth form on the boundary of $\sigma$. This extension however is not unique when $\sigma$ lies in the boundary of multiple simplices. In addition to the restriction of Whitney forms, we have the restriction for cochains.
If $f= \sum a_i \delta_{\sigma_i}$ is a cochain and $\sigma$ is a simplex, then $f|_{\sigma} = \sum\limits_{\sigma_i\subset\sigma}a_i\delta_{\sigma_i}.$ This cochain restriction satisfies $\omega|_{\sigma} = W(f|_{\sigma})|_{\sigma}$.

A Whitney form associated to a geodesic simplex $\sigma$ in $M$ is the corresponding Whitney form on the standard Euclidean simplex pulled back to $\sigma$ via the map $V_\sigma$. Geometric norms on cochains determined by the Whitney map can be compared with various combinatorial norms. This is done using the $L^p$-change of variables formula for $k$-forms, see \cite{stern}.

\begin{prop} \label{prop: L2 comparisons}
Let $s\in\mathcal S_\e\cup\mathcal G_\e$. Let $W(f)$ be the Whitney form associated to a cochain $f\in C^{k}(s;\R)$. Let $||\cdot||$ be some fixed norm on the real vector space $C^{k}(s;\R)$ and let $||\cdot||_{p,s}$ be the $p$-norm associated to $s$ on  $\Omega^k(s)$, where $p = {\infty}$ or $p=2$.
Then there is a constant $\mathcal A = \mathcal A (\e,||\cdot||)>0$ such that $$\mathcal A^{-1}||W(f)||_{p,{s}}\leq  \mathcal ||f|| \leq \mathcal A||W(f)||_{p,s}.$$ The constant $\mathcal A$ only depends on the chosen norm and combinatorial type of $s,$ not on the geometric structure of $s$.
\end{prop}
\begin{proof}
By Proposition \ref{prop: stars biLipschitz}, there is a constant $\mathcal L$ (note $\kappa\leq \mathcal L$, so that if we’re working with simplices, $\mathcal L$ works as Lipschitz constant) such that any pair $s,s’\in \mathcal S_\e\cup\mathcal G_\e$ in the same combinatorial type are $\mathcal L$-biLipschitz equivalent.
For each combinatorial type, fix a model $s_a\in \mathcal S_\e\cup\mathcal G_\e$ and let $\mu:s\to s_a$ be the $\mathcal L$-biLipschitz comparison map described above. Then, because Whitney forms are obtained by pulling back the standard Whitney forms via the model map, we can apply the $L^2$-change of variables formula for $k$-forms and use the biLipschitz comparison to get $$\frac{||W(f)||_{2,{s}}}{||f||}\leq \mathcal L^{n/2}\frac{||W(f)||_{2,s_a}}{||f||}.$$ Similarly, applying the $L^{\infty}$-change of variables formula for $k$-forms, gives
$$\frac{||W(f)||_{{\infty},s}}{||f||}\leq \mathcal L^{k}\frac{||W(f)||_{\infty,s_a}}{||f||}.$$

For $p=2$, set $$\mathcal A_2 =  \mathcal L^{n/2}\max\limits_a \sup\limits_{g\in C^{k}(s_a)} \frac{||W(g)||_{2,s_a}}{||g||}$$ and for $p = \infty$, set $$\mathcal A_{\infty} = \mathcal L^{n} \max\limits_a\sup\limits_{g\in C^{k}(s_a)} \frac{||W(g)||_{2,s_a}}{||g||},$$
where the maximum runs over all combinatorial types $a$.
Then, $A = \max\{\mathcal A_2,\mathcal A_{\infty}\}$ gives the first inequality.
The second inequality is obtained from an identical argument via the lower bound in biLipschitz comparison. Let $\mathcal A$ be the maximum of these two constants.
\end{proof}

This estimate is enough to reproduce the statements in \cite{LS} for deeply embedded triangulations. However, to obtain the cleaner discrete-smooth eigenvalue comparison in Proposition \ref{prop:4.2}, we need to study a smooth analogue of the Whitney map. The smooth Whitney map was introduced by Dodzuik in \cite{dodzuik2}. The map is defined by replacing the barycentric coordinates with a smooth partition of unity indexed by the vertices of a triangulation. The particular partition of unity is provided by the following proposition.

\begin{prop} \label{prop: partition existence} (\cite{dodzuik2}, Lemma 2.11)
If $M$ admits a deeply embedded triangulation $K$, then there exists a $C^{\infty}$ partition of unity $\beta_i$ indexed by the vertices of $K$ and subordinate to the covering of $M$ by open stars of vertices of $K$ (indeed, compactly supported in each open star). Moreover, each $\beta_i$ has covariant derivatives satisfying the pointwise bound $|\nabla^k\beta_i|< C$ for some constant $C = C(\e)$, for $k\leq n$.
\end{prop}

\begin{proof}
Let $s\in\mathcal S_\e$ be the combinatorial 1-neighborhood of a simplex. Denote the vertices of $s$ by $v_0,\dots, v_n$ and let $b_i$ be the standard barycentric coordinates associated to the vertex $v_i$. Define

\[\bar b_i(x) =  \begin{cases}
       0 & b_i(x)\leq 1/(n+2), \\
       \frac{(n+2)b_i(x)-1}{n+1} & b_i(x) \geq 1/(n+2). \\

   \end{cases}
\] Observe that $\sum\limits_i \bar b_i(y) \geq \frac{1}{(n+2)}$. Define $$\delta(s) = \inf\limits_{\substack{x\in\supp(\bar b_i)\\y\in \d \s(v_i)}}d(x,y),$$ where $d$ is the distance function induced by the Riemannian metric.  Set $\delta = \frac{1}{2}\inf\limits_{s\in \mathcal S_\e} \delta(s)$ and notice $\delta>0$. For any point $x\notin \s(v_i)$, the ball $B_{\delta}(x)$ is disjoint from the support of $\bar b_i$.
Let $\eta$ be a smooth cutoff function such that $\eta(r) = 1$ when $|r|<\delta/2$ and $\eta(r) = 0$ when $|r|>3\delta/4$. The function $\eta(d(x,y))$ is smooth on the open star of a vertex, so the operator given by integrating against $\eta(d(x,y))$ is smoothing. Therefore, if we define $$\tilde b_i(x) = \int_{B_{3\delta/4}(x)}\eta(d(x,y))\bar b_i(y)dy,$$ the result is a smooth function.  Notice $\tilde b_i$ is supported in the interior of the star of $v_i$ by virtue of our choice of $\delta$. We now define smoothed barycentric partitions of unity for a smooth manifold with deeply embedded triangulation $K$ by normalizing the functions $\tilde b_i$ associated to the vertices of $K$:
$$\beta_i(x) = \left(\sum\limits_j \tilde b_j (x)\right)^{-1} \tilde b_i (x).$$ Notice that if $\tilde b_i(x)\neq 0$, then this normalizing sum really just runs over the vertices of $\s(v_i)$. This normalizing constant can be bounded from below:
\begin{align*}
    \sum\limits_j \tilde b_j(x) &\geq \sum\limits_j \int_{B_{\delta/2}(x)}\bar b_j(y)dy \\
    &= \int_{B_{\delta/2}(x)} \sum\limits_j\bar b_j(y)dy \\
    &\geq \vol(B_{\delta/2}(x))\frac{1}{(n+2)(n+1)}.
\end{align*}

The covariant derivative bound follows from repeated application of the quotient rule and the corresponding bounds for $\bar b_i$, which  depends only on the derivatives of cutoff function $\eta$ and the covariant derivatives of the metric. The choice of $\delta$ ensures each function $\beta_i$ is compactly supported in the star of $v_i$.
\end{proof}

Our aim now is to establish a version of Proposition \ref{prop: L2 comparisons} for these partitions of unity that will allow us to relate the geometric norms induced by the smooth Whitney map to combinatorial norms.

Let $M$ be a hyperbolic $n$-manifold with a deeply embedded triangulation $K$. Denote by $\s_0(\sigma)$ the combinatorial 1-neighborhood of an $n$-simplex $\sigma$ in $K$.
Fix some $n$-simplex $\sigma$ in $K$ and set $s = \s_0(\sigma)$. Because $K$ is deeply embedded, for any point $p\in s$, the ball $B = B_\e(p)$ contains $s$ and lifts isometrically to $\H^n$. Identify $s$ with such a lift. The functions $\beta_i$ associated to the vertices of $\sigma$ are supported in $s$ and their value in any simplex $\sigma$ depends only on the geometry of $s$. This enables us to isolate the local properties of the barycentric partition of unity functions.
By perturbing the vertices of $s$ in $\H^n$ and modifying the various simplices making up the complex $s$ accordingly, we can then see how these functions relate to the geometry of combinatorial 1-neighborhoods as encoded by the space $\mathcal S_\e$.

\begin{lem} \label{lem: partition of unity continuous}
Let $\sigma$ be an $n$-simplex with vertices $v_0,\dots,v_n$ from $\mathcal G_\e$ contained in combinatorial 1-neighborhood $s = \s_0(\sigma) \in \mathcal S_\e$ with additional vertices $v_{n+1},\dots,v_m$.
The functions $\beta_i$ associated to the vertices of $\sigma$ constructed in Proposition \ref{prop: partition existence} and their covariant derivatives $\nabla \beta_i$ vary continuously in $L^2(\H^n)$ when the perturbation of $s$ in $\mathcal S_\e$ is realized as above in $\H^n$.
\end{lem}
\begin{proof}

As above, we can embed the combinatorial 1-neighborhood $s$ in $\H^n$ and the functions $\tilde b_i$ for each vertex $v_i$ of $\sigma$ described in Proposition \ref{prop: partition existence} are well defined smooth functions on $\H^n$ supported on $s$. To define the barycentric partition of unity, we also need the functions $\tilde b_k$ for vertices $v_k$ in $s$ that are not in $\sigma$ to be well defined on the support of the functions $\tilde b_i$ for vertices $v_i$ of $\sigma$.
For a vertex $v_k$ in $s$  that is not part of $\sigma$ and a point $x$ in the support of $\tilde b_i$ for $v_i$ a vertex of $\sigma$, we have that the ball $B_{3\delta/4}$ used in the definition of $\tilde b_k$ is contained in $s$, so $\tilde b_k(x)$ only depends on $s$.

From the definition, one sees that each function $\tilde b_i$ varies continuously in $L^2(\H^n)$ as the complex $s$ in $\H^n$ is varied by perturbing the vertices and modifying the various simplices making up the complex $s$ accordingly. One similarly can see from the definition of $\tilde b_i(x)$ that $\nabla \tilde b_i(x)$ varies continuously as $s$ is varied as above.
The functions $\beta_i$ in the barycentric partition of unity are defined by normalizing the functions $\tilde b_i$: $$\beta_i(x) = \left(\sum\limits_{v_k\in s^{(0)}} \tilde b_k (x)\right)^{-1} \tilde b_i (x) .$$ By the remark above, $\tilde b_k$ is well defined on the support of $\tilde b_i$ for every vertex $v_k$ of $s$. Since each $\tilde b_j$ and $\nabla \tilde b_j$ varies continuously with $s$, the same holds for $\beta_j$ and $\nabla \beta_j$.
\end{proof}

To a partition of unity indexed by the vertices of a triangulation and subordinate to the covering by open stars of vertices, one can define a generalized Whitney mapping, given by the same formula as the standard Whitney map but with the smooth barycentric partitions of unity in place of the standard barycentric coordinates. Let $\beta = (\beta_i)$ be the barycentric partition of unity. The Whitney form $W_{\beta}(\delta_{\sigma})$ associated to the cochain $\delta_{\sigma}$ dual to an  oriented simplex $\sigma = [v_0,\dots,v_q]$ is given by $$W_{\beta}(\delta_{\sigma}) = q! \sum_{k=0}^q(-1)^k\beta_{k}d\beta_0\wedge \cdots \wedge d \beta_{k-1}\wedge d \beta_{k+1}\wedge \cdots \wedge d \beta_q,$$ Like the standard Whitney map, these generalized Whitney maps satisfy:
\begin{enumerate}
	\item For a chain $a$ and cochain $f$ of the same degree, $\int_aW_{\beta}(f) = f(a).$
	\item For any cochain $f$, $dW_{\beta}(f) = W_{\beta}(df)$.
	\item If $p$ is contained in the interior of an  $n$-simplex $\sigma$, and any cochain $f$, $W_{\beta}(f)_p = W_{\beta}(f|_{\sigma})_p$.
\end{enumerate}

\begin{lem} \label{lem: forms continuous} Fix a simplex $\sigma\in\mathcal G_\e$ contained in its combinatorial 1-neighborhood $s\in \mathcal S_\e$.
The Whitney map $W_{\beta}: C^{\bullet}(\sigma)\to L^2\Omega^{\bullet}(\H^n)$ varies continuously as the geometry of the star $s$ varies in $\H^n$ as described above. Consequently, $||W_{\beta}(f)||_{2,s}$ and $|| W_{\beta}(f) ||_{2,\sigma}$ vary continuously with the geometric structure on $s$ in $\mathcal S_\e$.
 \end{lem}
\begin{proof} We work with the combinatorial 1-neighborhood $s$ embedded in a ball $B\subset \H^n$ as above. Because $\beta_i$ and $\nabla \beta_i$ vary continuously with $s$ and since $||\nabla\beta_i||_{2,B} $ is comparable to $||d\beta_i||_{2,B}$, it suffices to show that exterior products of $d\beta_i$ vary continuously in the $L^2$-norm. The degree 0 case and  degree 1 case are immediate from the continuity in Lemma \ref{lem: partition of unity continuous}. We treat only the degree 2 case as the other higher degree cases are handled similarly. Assume $||\beta_i - \beta_i’||_{2,B} < \epsilon$ and $||d\beta_i - d\beta_i’||_{2,B}<\e$. Then we have \begin{align*}
||d\beta_0\wedge d\beta_1 - d\beta’_0\wedge d\beta_1’||_{2,B} &=||d\beta_0\wedge d\beta_1 - d\beta’_0\wedge d\beta_1’ + d\beta_0 \wedge d\beta_1’ - d\beta_0 \wedge d\beta_1’||_{2,B} \\
&\leq ||d\beta_0\wedge d\beta_1  - d\beta_0 \wedge d\beta_1’||_{2,B} + || d\beta_0’ \wedge d\beta_1’ - d\beta_0\wedge d\beta_1’||_{2,B} \\
&= ||d\beta_0\wedge d(\beta_1 - \beta_1’)||_{2,B} + ||d(\beta_1’ - \beta_1)\wedge d\beta_1’||_{2,B}\\
&\leq C||d\beta_0||_{2,B} ||d(\beta_1-\beta_1’)||_{2,B} + C||d\beta_1’
||_{2,B} ||d(\beta_1’-\beta_1)||_{2,B},\\
&\leq 2C\epsilon, \text{~after increasing $C$.}
\end{align*}
This immediately gives that $||W_{\beta}(f)||_{2,s}$ varies continously with $s$. Let $\xi_{\sigma}$ be the characteristic function of $\sigma$ in $\H^n$. Then $\xi_{\sigma}$ varies continuously in $L^2$ when $\sigma$ is varied by perturbing its vertices. The norm $||W_{\beta}(f)||_{2,\sigma} = ||W_{\beta}(f)\xi_{\sigma}||_{2,B}$ therefore also varies continuously when $s$ is varied in $\mathcal S_\e$.
\end{proof}

\begin{prop} \label{prop: smooth L2 comparison} There is a constant $A=A(\e)>0$ such that if $\sigma\in \mathcal G_\e$ has combinatorial 1-neighborhood $s\in\mathcal S_\e$ and $f\in C^{\bullet}(\sigma)$ is any cochain, then there are comparisons $$ A^{-1}||W(f)||_{2,\sigma}\leq ||W_{\beta}(f)||_{2,\sigma}\leq A||W(f)||_{2,\sigma}.$$
\end{prop}
\begin{proof}
    We embed $s$ in a ball $B$ in $\H^n$ as above. The $L^2$-norm induced by $\beta$ for a fixed geometric structure on $\sigma$ is continuous on the vector space of cochains $C^{\bullet}(\sigma)$. For any cochain $f\in C^{\bullet}(\sigma)$, the form $W_{\beta}(f)$ varies continuously in $L^2\Omega^{\bullet}(B)$ as the combinatorial 1-neighborhood $s$ varies in $\mathcal S_\e$.
    It follows that $||W_{\beta}(f)||_{2,\sigma}$ is continuous as a function on the component of $\mathcal S_\e \times C^{\bullet}(\sigma)$ corresponding to the combinatorial type of the combinatorial 1-neighborhood $s$.

    Since $W_{\beta}$ sends nonzero cochains to nonzero forms, for any $f\neq 0$, one has $0<||W_{\beta}(f)||_{2,\sigma}$. Let $||\cdot||_{E}$ be the usual $\ell^2$ norm on $C^{\bullet}(\sigma)$.
    The norm $\norm{\cdot}_E$ is fixed as the geometric structure on $\sigma$ varies. For $s’\in \mathcal S_\e$, let $\beta’$ be the corresponding  barycentric partition of unity defined by the combinatorial 1-neighborhood $s’$. By definition, each $s’\in\mathcal S_\e$ is the combinatorial 1-neighborhood of some simplex, let $\sigma’$ be this simplex.

    Using the continuity in Lemma \ref{lem: partition of unity continuous} and the compactness of $\mathcal S_\e$, we conclude the constants
    $$A_{\bullet} = \inf\limits_{s’\in \mathcal S_\e}\inf\limits _{\substack{f\in C^{\bullet}(\sigma’)\\ ||f||_{E} = 1}} ||W_{\beta’}(f)||_{2, \sigma’}$$
    and
    $$B_{\bullet}= \sup\limits_{s’\in \mathcal S_\e}\sup\limits_{\substack{f\in C^{\bullet}(\sigma’)\\ ||f||_{E} = 1}} ||W_{\beta’}(f)||_{2, \sigma’},$$
    are strictly positive real numbers independent of the particular geometric structure on $s$ or $\sigma$ giving the desired comparison between $||f||_E$ and $||W_{\beta}(f)||_{2,\sigma}$.

    We can then compare $||W(f)||_{2,\sigma}$ and $||f||_E$ using Proposition \ref{prop: L2 comparisons}. For $f\in C^{\bullet}(\sigma)$, Proposition \ref{prop: L2 comparisons} gives a constant $\mathcal A$ such that $\mathcal A^{-1}||W(f)||_{2,\sigma} \leq ||f||_{E}\leq \mathcal A||W(f)||_{2,\sigma}$.
    Combining these comparisons with the comparison of $||W_{\beta}(f)||_{2,\sigma} $ and $||f||_{E}$ above gives the claimed result for $||W_{\beta}(f)||_{2,\sigma} $.
\end{proof}

The upshot of this is that the smooth barycentric partition of unity induces a norm on the cochain complex of a simplex that locally is uniformly comparable to the $L^2$-norm induced by the standard barycentric partition of unity.  This gives an analogue of Proposition \ref{prop: L2 comparisons} for the $L^2$-norm induced by the barycentric partition of unity. We also require such a comparison for the $L^{\infty}$-norm. The upgraded version of Proposition \ref{prop: L2 comparisons} appears below as Proposition \ref{prop: smooth comparison}.

For a smooth manifold $Y$, possibly with boundary, we denote the Sobolev spaces of differential $\bullet$-forms by $H^k_{\nabla}\Omega^{\bullet}(Y)$ and their norms by $||\cdot||_{H^k_{\nabla}(Y)}$, where $$||\omega||_{H^k_{\nabla}} = \sum\limits_{i=0}^k||\nabla^i \omega||_2.$$ When $\bullet = 0$, we drop $\Omega^0$ from this notation.
When $Y$ has boundary, the marking $\mathring{H}^k_{\nabla}$ denotes the subspace of forms that are approximated by smooth forms supported in the interior of $Y$.

\begin{lem} \label{lem: constant R}
    There is a constant $R(\e)>0$ such that for any $n$-simplex $\sigma\in \mathcal G_\e$ with combinatorial 1-neighborhood $s\in\mathcal S_\e$, the map $W_{\beta}:C^{\bullet}(\sigma)\to H^{n}_{\nabla}\Omega^{\bullet}(B)$ satisfies
    $$||W_{\beta}(f)|_{\sigma}||_{H^n_{\nabla}(\sigma)} \leq ||W_{\beta}(f)||_{H^n_{\nabla}(B)} \leq R||W_{\beta}(f)||_{2,\sigma},$$ where we have identified the combinatorial 1-neighborhood $s$ isometrically with a domain $D$ in $\H^n$ and $B$ is a ball of radius $\e$ based at $p\in \sigma$.
\end{lem}

\begin{proof}
The first inequality follows from the definition of the Sobolov norm.

We now observe that the covariant derivative bounds for a smooth barycentric partition of unity imply that $||W_{\beta}(f)||_{H^n_{\nabla}(B)}$ is bounded by some constant times $||f||_{G,\sigma}$, where $||\cdot||_{G,\sigma}$ is the $\ell^1$-norm on $C^{\bullet}(\sigma)$. For a cochain $f = \sum a_i \delta_{\sigma_i}$, let $\omega_i = W_{\beta}(\delta_{\sigma_i})$ so that $W_{\beta}(f) = \sum a_i \omega_i$. We can then compute,
\begin{align*}
	||W_{\beta}(f)||_{H^k_{\nabla}(B)} &= ||\sum a_i\omega_i||_{H^k_{\nabla}(B)}\\
	&\leq \sum_j||\nabla^j \sum_i  a_i\omega_i||_{2,B}\\
	&\leq  \sum_j\sum_i|a_i|||\nabla^j \omega_i||_{2,B}.
\end{align*}
Each summand above satisfies $||\nabla^j\omega_i||_{2,B}<C$ for a constant $C$ depending on the covariant derivative bounds of the barycentric partition of unity. There is a constant $T$ such that the number of $\bullet$-faces of an $n$-simplex is less than $T$.
Thus, \begin{align*}
 ||W_{\beta}(f)||_{H^k_{\nabla}(B)} &\leq \sum_j\sum_i|a_i|||\nabla^j \omega_i||_{2,B} \\
&\leq \sum_jC\sum_i|a_i|\\
&\leq TC\sum_i|a_i| \\
&= TC||f||_{G,\sigma}.
\end{align*}
 For a single simplex, this combinatorial $\ell^1$-norm is comparable to the $\beta$-induced $L^2$-norm by Proposition \ref{prop: L2 comparisons} and Proposition \ref{prop: smooth L2 comparison}. Thus, $$||W_{\beta}(f)||_{H^n_{\nabla}(B)}\leq R||W_{\beta}(f)||_{2,\sigma}.$$
\end{proof}

The following version of the Sobolev inequality, a consequence of Theorem 1 in \cite{cantor}, ensures that control over certain Sobolev norms implies pointwise norm control. Let $|\cdot|$ be the pointwise norm induced by the Riemannian metric.

\begin{thm} \label{thm: Cantor Sobolev}(Cantor-Sobolev)
Suppose $M$ is a hyperbolic $n$-manifold with injectivity radius bounded below by $\e$. Let $r<\e$, and $l\geq 0,~k\geq 0$ be such that $l + n/2 < k$. Then if $\omega$ is in the Sobolev space $H^k_{\nabla} \Omega^{\bullet}(M)$, there is a constant $C = C(r)$ such that for every $p\in M$, $$|\nabla^l\omega(p)|\leq C||\omega||_{H_{\nabla}^k(B_r(p))}.$$
\end{thm}

This discussion provides us with the following upgraded version of Proposition \ref{prop: L2 comparisons}.

\begin{prop} \label{prop: smooth comparison}
Let $s\in\mathcal S_\e$ be the combinatorial 1-neighborhood of an $n$-simplex $\sigma \in\mathcal G_\e$. Let $\beta$ be the smooth barycentric partition of unity in Proposition \ref{prop: smooth L2 comparison} or the standard barycentric coordinate on $s$. Let $W_{\beta}(f)$ be the resulting generalized Whitney form associated to a cochain $f\in C^{\bullet}(\sigma;\R)$. Let $||\cdot||$ be some fixed norm on the real vector space $C^{\bullet}(\sigma;\R)$ and let $||\cdot||_{p,s}$ be the $p$-norm associated to $s$ on $\Omega^{\bullet}(s)$ and likewise for $||\cdot||_{p,\sigma}$, where $p = {\infty}$ or $p=2$.
Then there is a constant $\mathcal B = \mathcal B (\e,||\cdot||)>0$ (independent of $\sigma$ and $s$) such that $$\mathcal B ^{-1}|| W_{\beta}(f)||_{2,{\sigma}}\leq   ||f||\leq \mathcal B|| W_{\beta}(f)||_{2,\sigma}$$ and
$$\mathcal B^{-1}|| W_{\beta}(f) ||_{\infty, \sigma}\leq \mathcal B^{-1} || W_{\beta}(f)||_{\infty,s} \leq ||f||\leq \mathcal B|| W_{\beta}(f)||_{\infty,s}.$$
\end{prop}

\begin{proof}
    Identify $s$ with a domain $D$ in $\H^n$ and let $B$ be a ball of fixed radius $r = r(\e)$ that contains $s$ as in Lemma 2.11 and assume the basepoint of the ball is the point at which $||W_{\beta}(f)||_{\infty, s}$ is realized.

    From the Sobolev inequality, Propositions \ref{prop: L2 comparisons} and \ref{prop: smooth L2 comparison}, and Lemma \ref{lem: constant R}, for any cochain $f$,
        \begin{align*}
        ||W_{\beta}(f)||_{\infty,\sigma}&\leq ||W_{\beta}(f)||_{\infty,s}\\
        &\leq C||W_{\beta}(f)||_{H^n_{\nabla}(B)}\\
        &\leq CR||f||_{G,\sigma} \\
        &\leq \mathcal ACR||W_{\beta}(f)||_{2,\sigma}.
        \end{align*}

    Then the comparison $$||W_{\beta}(f)||_{2,\sigma}\leq ||W_{\beta}(f)||_{2,s}\leq\sqrt{\vol(s)} ||W_{\beta}(f)||_{\infty,s}$$ implies the smooth barycentric partition of unity induced $L^2$-norm $||W_{\beta}(f)||_{2,\sigma} $ and $L^{\infty}$-norm $||W_{\beta}(f)||_{\infty,s}$ are comparable for any simplex $\sigma\in\mathcal G_\e$ with combinatorial 1-neighborhood
    $s\in \mathcal S_\e$.

    The only constant appearing in the comparison depending on $s$ is $\sqrt{\vol(s)}$, which can be uniformly bounded by a constant depending only on $\e$ by compactness of $\mathcal S_\e$. As remarked after the proof of Proposition \ref{prop: smooth L2 comparison}, the $L^2$-norm version of the desired estimate follows from Proposition \ref{prop: L2 comparisons} and Proposition \ref{prop: smooth L2 comparison}. The above comparisons of the $L^2$-norm and the $L^{\infty}$-norm imply the claim.
\end{proof}

\section{Norm estimates}
\label{sec:3}

In this section, we use deeply embedded triangulations and the Whitney maps described in the previous section to compare various geometric and combinatorial norms on forms and cochains. Throughout, let be $M$ a closed hyperbolic manifold of dimension $n>2$ with injectivity radius bounded below by $\e>0$ and a fixed deeply embedded triangulation $K$. Let $\beta$ be the smooth barycentric partition of unity associated to $K$.

We require various comparisons of the following norms on cochain and chain complexes associated to $M$ and $K$. The relevant norms are:

\begin{enumerate}
    \item The combinatorial Gromov norm $||\cdot||_G$ on any chain or cochain complex given by $||\sum a_i\sigma_i||_G = \sum|a_i|$ and $||\sum a_i\delta_{\sigma_i}||_G = \sum|a_i|$.
    \item The combinatorial Euclidean norm $||\cdot||_E$, which is the usual $\ell^2$ norm on chains and cochains given by $ ||\sum a_i\sigma_i||_E = \sqrt{\sum|a_i|^2}$ and  $ ||\sum a_i\delta_{\sigma_i}||_E = \sqrt{\sum|a_i|^2}$.
     \item The combinatorial max norm $||\cdot||_{\max}$ on any chain or cochain complex given by $||\sum a_i\sigma_i||_{\max} = \max|a_i|$.

    \item The Whitney induced $L^2$-norm $||~\cdot~||_2$ on the cochain complex $C^{\bullet}(K)$, given by $||f||_2 = \sqrt{\int_MW_{\beta}(f)\wedge\star W_{\beta}(f)}$.

    \item The Whitney induced $L^{\infty}$-norm $||\cdot||_{\infty}$ on the cochain complex $C^{\bullet}(K)$, given by taking the essential supremum of the pointwise Riemannian metric operator norms $||f||_{\infty} = \esssup\limits_{p\in M} ||W_{\beta}(f)_p||_{\infty}.$
\end{enumerate}

Given a norm $||\cdot||$ on the cochain complex $C^{\bullet}(K)$, let $||\cdot||^*$ denote the dual norm on the linear dual chain complex $C_{\bullet}(K)$ induced by the integration pairing: $$||a||^* = \sup\limits_{\substack{||f||=1\\ f\in C^{\bullet}(K)}} \int_a W_{\beta}(f).$$

\begin{prop} \label{prop: 3.1} There is a constant $B =  B(\e)>0$ such that the norms $||\cdot||_G$ and $||\cdot||_2$ on $C^{\bullet}(K)$ satisfy $$||\cdot||_G\leq B\sqrt{\vol(M)}||\cdot||_2,$$ and the norms $||\cdot||_G$ and $||\cdot||_2^*$ on $C_{\bullet}(K)$ satisfy  $$||\cdot||_G\leq B \sqrt{\vol(M)}||\cdot||_2^{*}.$$
\end{prop}
\begin{proof}

 Let $f = \sum\limits_Fa_F\delta_F$ be a cochain.
 Then for any $n$-simplex $\sigma$, $f|_{\sigma} = \sum\limits_{F\subset \sigma}a_F\delta_F$ and $$||f||^2_2 = \sum\limits_{\sigma \in K^{(n)}} || W_{\beta}(f|_{\sigma})|_{\sigma}||_2^2 =\sum\limits_{\sigma \in K^{(n)}} ||W_{\beta}(f)|_{\sigma}||_2^2 .$$
Apply Proposition \ref{prop: smooth comparison} to obtain $D = \mathcal B(\e,||\cdot||_G)$. This gives $||f|_{\sigma}||_G\leq D ||f|_{\sigma}||_{2,\sigma},$ where $||\cdot||_{2,\sigma}$ is the $L^2$-norm on the simplex $\sigma$ associated to the smooth barycentric coordinate $\beta$.
Then, by applying the Euclidean $\ell^1$-$\ell^2$-comparison to the cochain complex and using the fact there is a constant $T$ such that the number of $n$-simplices in $K$ is less than $T\vol(M)$, we find
\begin{align*}
||f||_G  &\leq \sum\limits_{\sigma \in K^{(n)}} ||f|_{\sigma}||_G\\
		&\leq \sum\limits_{\sigma \in K^{(n)}}  D||f|_{\sigma}||_{2,\sigma}\\
	   &\leq D\sqrt{T\vol(M)\sum\limits_{\sigma \in K^{(n)}}||W_{\beta}(f)|_{\sigma}||_2^2}\\
	   &\leq D\sqrt{T\vol(M)}||f||_2.
\end{align*}
For the second inequality, notice $||\cdot||_G$ is the usual $\ell^1$-norm on a finite dimensional vector space, so its dual norm is the max norm $||\cdot||_{\max}$.
\color{black}

Apply Proposition \ref{prop: smooth comparison} and set $ D’ =\mathcal B(\e, ||\cdot||_{\max}),$ so that if $\sigma$ is the simplex in which $||f||_{\infty}$ is realized, then $$||f||_{\infty} = ||f|_{\sigma}||_{\infty}\leq  D’||f|_{\sigma}||_{\max}\leq D’||f||_{\max}.$$
Then, $$||f||_2 \leq\sqrt{\vol(M)} ||f||_{\infty} \leq D’ \sqrt{\vol(M)}||f||_{\max}.$$
Dualizing gives $$  ||\cdot||_G = ||\cdot||_{\max}^*\leq  D’ \sqrt{\vol(M)}||\cdot||_2^*,$$
since
\begin{align*}
	||a||_G &= ||a||_{\max}^*\\
		   &= \sup\limits_{||f||_{\max} \leq 1}\int_a W_{\beta}(f)\\
		   &=\sup\limits_{|| D’\sqrt{\vol(M)}f||_{\max} \leq 1}\int_{a} D’\sqrt{\vol(M)} W_{\beta} (f)\\
		   &\leq \sup\limits_{||f||_2\leq1}\int_{a} D’\sqrt{\vol(M)} W_{\beta} (f)\\
		   &=  D’\sqrt{\vol(M)} \sup\limits_{||f||_2\leq 1} \int_a W_{\beta} (f)\\
		   &= D’\sqrt{\vol(M)}||a||_2^*.
\end{align*}
Set $B = \max\{D\sqrt{T}, D’\}$ to obtain the claim.
 \end{proof}

Recall from Section 2 that there is a polyhedral celluation $K^*$ dual to $K$ that can be canonically subdivided into a triangulation $\tau(K)$. Equipping these dual complexes with the Gromov norm, we have the following two propositions relating these norms by the Poincar\'e duality and subdivision maps.

\begin{prop} \label{prop: 3.2} 
There is a constant $D = D(\e)$ such that for any $\bullet$-cochain $f\in C^{\bullet}(K)$ one has $||f||_2 \leq D||f||_G$.
\end{prop}

\begin{proof} Let $f = \sum a_i\delta_{\sigma_i}$ be a $\bullet$-cochain. Then $||\omega||_G = \sum|a_i|$ and $||f||_2 \leq \sum |a_i| ||\delta_{\sigma_i}||_2.$ Then, for any fixed $\bullet$-simplex $\sigma$ that is a face of an $n$-simplex from $\mathcal G_e$, using the $L^2$-change of variables formula and Proposition \ref{prop: smooth L2 comparison} we can take $D = A\mathcal L^{n/2}||W(\delta_{\sigma})||_2$, so that $||\delta_{\sigma_i}||_2 \leq D.$

The comparison \[||f||_2\leq\sum |a_i| ||\delta_{\sigma_i}||_2 \leq D\sum |a_i| = D||f||_G \] then follows.
\end{proof}

\begin{prop} \label{prop: 3.3} 
The Poincar\'e duality map $\Phi:C^{\bullet}(K)\to C_{n-\bullet}(K^*)$ preserves the Gromov norm $$||f||_G= ||\Phi(f)||_G.$$
\end{prop}

\begin{proof}
  Let $f = \sum a_i\delta_{\sigma_i}$, then $\Phi(f) = \sum a_i (\sigma_i)^*$.
\end{proof}

\begin{prop} \label{prop: 3.4} 
Let $N$ be the constant from Proposition \ref{prop: star bound}, which bounds the number of simplices in the star of a simplex in a deeply embedded triangulation. Then the subdivision map $\tau: C_{2}(K^*)\to C_{2}(T)$ satisfies $$ || \tau(c)||_G \leq N ||c||_G.$$
\end{prop}

\begin{proof}
  The number of sides of a $2$-cell in $K^*$ dual to a $(n-2)$-simplex $\sigma$ in $K$ corresponds to the number of $n$-simplices in $K$ that contain $\sigma$. The number of such simplices is bounded by $N$.
\end{proof}

The following estimates are essential in comparing the first eigenvalue of the Whitney Laplacian to the genuine first eigenvalue. For this, we need to work with various Sobolev spaces to control the orthogonal projection of a Whitney form onto its coexact part. This discussion is the reason we use the smoothed Whitney forms in place of the standard ones.

We will require the following version of the Gaffney inequality, which follows from Lemma 2.4.10 in \cite{schwarz}. To simplify the following discussion, for a smooth manifold $Y$ possibly with boundary, we introduce an alternative Sobolev norm on $H^{k+1}_{\nabla}\Omega^{\bullet}(Y)$: $$||\omega||_{A^{k}(Y)}:= ||\omega||_{H^k_{\nabla}(Y)} + ||d\omega||_{H^k_{\nabla}(Y)} + ||d^*\omega||_{H^k_{\nabla}(Y)}.$$

Since $d$ and $d^*$ are bounded operators $H^{k+1}\Omega^{\bullet}_{\nabla}(Y)\to H^{k}\Omega^{\bullet\pm 1}_{\nabla}(Y)$, we immediately have that there is a constant $C$ such that $||\omega||_{A^{k}(Y)}\leq  C ||\omega||_{H^{k+1}_{\nabla}(Y)}.$

Recall that the marking $\mathring H$ denotes the subspace of given Sobolev space that is the closure of smooth functions supported in the interior.

\begin{lem}  \label{lem: 3.5} (Gaffney inequality) Let $Y$ be a smooth manifold with boundary. Let $\omega\in \mathring{H}^k_{\nabla} \Omega^{1}(Y)$. Then there is a constant $C = C(Y)>0$ such that $||\omega||_{\mathring{H}^k_{\nabla}(Y)}\leq C ||\omega||_{A^{k-1}(Y)}$.
\end{lem}

\begin{lem} \label{lem: 3.6} 
Let $B_0 = B_r(p)$ and $B_1= B_{r+\delta}(p)$ be a pair of concentric balls in $\H^n$ and let $\phi$ be a smooth bump function that is identically 1 on $B_0$ and vanishes in a neighborhood of $\d B_1$. There is a constant $C = C(\phi,k)$ that depends only on the norm of the covariant derivatives of $\phi$ up to order $k+1$ such that if $\omega \in \Omega^{1}(\H^n)$, then $$||\phi\omega||_{A^k(B_1)}\leq C ||\omega||_{A^k(B_1)}.$$
\end{lem}

\begin{proof}
Notice that $d\phi\omega = d\phi\wedge\omega + \phi d \omega$ and $d^*\phi\omega = \phi d^*\omega + g(\nabla \phi, X_{\omega}),$ where $X_{\omega}$ is the vector field dual to $\omega$.
As a result, the triangle inequality yields \begin{align*}
||\phi\omega||_{A^k(B_1)} &\leq ||\phi d\omega||_{H^k_{\nabla}(B_1)} + ||\phi d^*\omega||_{H^k_{\nabla}(B_1)}+ ||\phi\omega||_{H^k_{\nabla}(B_1)} \\ &+ ||d\phi\wedge\omega||_{H^{k}_{\nabla}(B_1)} + ||g(X_{\omega},\nabla\phi)||_{H^k_{\nabla}(B_1)}.
\end{align*}

The estimate $||\alpha\wedge\beta||_2\leq ||\alpha||_{\infty}||\beta||_2$ implies $$||d\phi\wedge\omega||_{H^k_{\nabla}(B_1)} \leq C||\omega||_{H^k_{\nabla}(B_1)},$$ where the constant $C$ is given by the sum of the $||\cdot||_{\infty}$-norms of the covariant derivatives of the bump function $\phi$.
The same argument gives for any form $\xi$ that $||\phi\wedge\xi||_{H^k_{\nabla}(B_1)}\leq C ||\xi||_{H^k_{\nabla}(B_1)}$. Applying this estimate to $\phi d\omega, \phi d^*\omega$, and $\phi \omega$ handles all terms in the comparison save for $||g(X_{\omega},\nabla\phi)||_{H^{k}_{\nabla}(B_1)}$.
For this term, notice that $\nabla g(X_{\omega},\nabla\phi) = g(\nabla X_{\omega},\nabla \phi) + g( X_{\omega}, \nabla^2\phi).$ We therefore have the pointwise estimate

\begin{align*}
|\nabla g(X_{\omega},\nabla\phi) | &\leq |g(\nabla X_{\omega},\nabla \phi)| + |g( X_{\omega}, \nabla^2\phi)| \text{ by the triangle inequality,} \\
&\leq |\nabla \omega|^2|\nabla\phi|^2 + |\omega|^2|\nabla^2\phi|^2,
\end{align*}

by applying Cauchy-Schwarz and the musical isomorphism.
Integrating then gives the corresponding inequality for the Sobolev norm $H^1_{\nabla}$. Repeating this calculation for higher order covariant derivatives completes the proof.
\end{proof}

The two previous lemmas combine to give the following statement.
\begin{prop} \label{prop: 3.7} 
Let $B_0 = B_r(p)$ and $B_1= B_{r+\delta}(p)$ be a pair of concentric balls in $\H^n$ Then there is a constant $C = C(r,\delta)$ such that for any $\omega \in H^k_{\nabla}(B_1)$ one has $$||\omega||_{H^k_{\nabla}(B_0)}\leq C||\omega||_{A^{k-1}(B_1)}.$$
\end{prop}
\begin{proof}
Let $\sqrt{C}$ be the maximum of the constants from Gaffney’s inequality and Lemma 3.6 with bump function $\phi$. Then for a form $\omega \in H^k_{\nabla}(B_1)$, one has $$||\omega||_{H^k_{\nabla}(B_0)}\leq ||\phi\omega||_{H^k_{\nabla}(B_1)},$$ since $\phi\omega|_{B_0} = \omega.$
Since $\phi\omega$ vanishes on $\d B_1$, Gaffney’s inequality and Lemma 3.6 give $$||\phi\omega||_{H^k_{\nabla}(B_1)} = ||\phi\omega||_{\mathring{H}^k_{\nabla}(B_1)}\leq \sqrt{C} ||\phi\omega||_{A^{k-1}(B_1)}\leq C||\omega||_{A^{k-1}(B_1)}.$$ Combining these two estimates gives the proposition.
\end{proof}

\begin{prop}\label{prop: 3.8} 
Let $M$ be a hyperbolic $n$-manifold with injectivity radius greater than $\e>0$. There is a constant $H(\e) = H >0$ depending only on $\e$ such that if the $L^2$-Hodge decomposition of a smooth 1-form $\omega\in \Omega^1$ has coexact part $\alpha$, then for any point $p\in M$ there is a ball $B\subset M$ centered at $p$ of radius determined by $\e$ such that for $k\geq n$ one has $$|
\nabla \alpha(p)|\leq H\left(||\omega||_{H^k_{\nabla}(B)} + ||\alpha||_{2,B}\right)$$ and consequently $$|
\nabla \alpha(p)|\leq H||\omega||_{H^k_{\nabla}(M)}.$$
\end{prop}

\begin{proof}

The $L^2$-Hodge decomposition of $M$ determines an orthogonal decomposition of $\omega= \alpha + \eta + h$, where $\alpha = d^*A$, $\eta = db$, and $h$ harmonic. Cantor’s estimate (Theorem \ref{thm: Cantor Sobolev} implies at any point $p$ of $M$, $$|\nabla \alpha (p)| \leq C(r)||\alpha||_{H^k_{\nabla}(B_r(p))}$$ so long as $r<\inj(M)$ and $k > n/2 + 1$. Since we assume dimension $n>2$, any $k\geq n$ suffices.
Take a concentric family of balls $B_i = B_{r + i\delta}(p)$ where $r = 6\e$, and $k\delta + r < 10\e < \inj(M)$; note that $B_0$ contains the 0-star of $\sigma$. Let $\phi_i$ be a bump function that is identically one on $B_i$ and vanishes on $\d B_{i+1}$.

Letting $C_i$ be the maximum of the constant from Cantor’s estimate on the ball $B_i$ and the constant from Proposition \ref{prop: 3.7} for the balls $B_i\subset B_{i+1}$, we have
$$|\nabla \alpha (p)| \leq C_i||\alpha||_{H^k_{\nabla}(B_i)}
   \leq  C_i^2||\alpha||_{A^{k-1}(B_{i+1})}.$$
Since $\alpha$ is coexact, $$||\alpha||_{A^{k-1}(B_{i+1})} = ||\alpha||_{H^{k-1}_{\nabla}(B_{i+1})} + ||d\alpha||_{H^{k-1}_{\nabla}(B_{i+1})}.$$ Notice that $d\alpha = d\omega$, and that $d:H^k_{\nabla}(B_i)\to H^{k-1}_{\nabla}(B_i)$ is a bounded operator, say with operator norm $C$. Thus, we have $$||d\alpha||_{H^{k-1}_{\nabla}(B_{i+1})} \leq C||\omega||_{H^k_{\nabla}}.$$ As a result, we get the estimate $$||\alpha||_{A^{k-1}(B_{i+1})} \leq ||\alpha||_{H^{k-1}_{\nabla}(B_{i+1})} + C||\omega||_{H^{k}_{\nabla}(B_{i+1})}$$
Combining this with the estimate of $\nabla \alpha$ above gives $$|\nabla \alpha(p)|\leq C_i\left(||\omega||_{H^{k}_{\nabla}(B_{i+1})}+||\alpha||_{H^{k-1}_{\nabla}(B_{i+1})}\right).$$
We can repeat argument this using the family of balls $B_i$ to reduce the order of the Sobolev norm of $\alpha$ on the right-hand side until we obtain $$|\nabla \alpha(p)|\leq H\left(||\omega||_{H^k_{\nabla}(B_n)} + ||\alpha||_{2,B_n}\right),$$ where $H$ is obtained by combining all the constants appearing in the iterated calculation. Orthogonality of the Hodge decomposition implies $||\alpha||_2\leq ||\omega||_2$, and we clearly have $||\omega||_{2, B_n}\leq ||\omega||_{H^k_{\nabla}(M)},$ so we are done after increasing $H$ by 1.
\end{proof}

When the form in the previous proposition is a Whitney form, we can make the following refinement.

\begin{prop}\label{prop: 3.9} 
Let $M$ be a hyperbolic $n$-dimensional manifold with a deeply embedded triangulation $K$ and let $f\in C^1(K)$. There is a constant $H(\e) = H >0$ depending only on $\e$ such that if the $L^2$-Hodge decomposition of the generalized Whitney form $W_{\beta}(f)$ has coexact part $\alpha$,  then  $$
||\nabla \alpha||_{\infty}\leq H||W_{\beta}(f)||_2.$$
\end{prop}
\begin{proof}
Let $\omega = W_{\beta}(f)$ be a smooth Whitney form. We will apply Proposition \ref{prop: 3.8} at a point $p$ at which $||\nabla\alpha||_{\infty}$ is realized. Proposition \ref{prop: 3.8} gives a ball $B$ about $p$ of radius depending only on $\e$ such that $$|
\nabla \alpha(p)|\leq H\left(||\omega||_{H^k_{\nabla}(B)} + ||\alpha||_{2,B}\right),$$ for a constant $H$ depending only on $\e$.
The ball $B$ intersects some uniformly bounded collection of $n$-simplices $\sigma’$ from $K$ where the constant depends only on $\e$; let $T = T(\e)$ be this bound. We can therefore estimate the norm of the Whitney form term by $$||\omega||_{H^k_{\nabla}(B)}\leq \sum\limits_{\sigma’\cap B \neq \emptyset} ||\omega|_{\sigma’}||_{H^k_{\nabla}(\sigma’)}.$$
Applying Lemma 2.13 to each summand in the previous estimate then gives
$$ ||\omega||_{H^k_{\nabla}(B)} \leq \sum\limits_{\sigma’\cap B_i \neq \emptyset} ||\omega|_{\sigma’}||_{H^k_{\nabla}(\sigma’)}\leq R\sum\limits_{\sigma’\cap B \neq \emptyset} ||\omega|_{\sigma’}||_2\leq R\sqrt{T}||\omega||_2.$$
Combining the above then gives that $$|\nabla \alpha(p)|\leq HR\sqrt{T}||\omega||_2 + H||\alpha||_{2,B}.$$ Clearly $$H||\alpha||_{2,B}\leq H||\alpha||_2 \leq H||\omega||_2,$$ so that after increasing $H$ to absorb the $R\sqrt{T}$ term we are done.
\end{proof}

\begin{prop}\label{prop: 3.10} 
Let $M$ be a hyperbolic $n$-manifold. Let $\omega\in \Omega^1(M).$ Assume there exists a constant $H$ such that $|\nabla \omega| \leq H$. Then there is a constant $C(H,\e)$ such that $||\omega||_{\infty}\leq C(H,\e) ||\omega||_2.$
\end{prop}

\begin{proof} Assume $||\omega||_{\infty} = 1$ and is realized at the point $p$. By Kato’s inequality and the hypothesis, away from the zeros of $\omega$ one has $|\nabla|\omega||\leq |\nabla \omega| \leq H$. Fix a normal coordinate frame $x_0,\dots,x_{n-1}$ at $p$ of radius $2\e$. Define the function $\phi$ on this normal coordinate neighborhood by $\phi(x) = 1- Hd(x,p)$ for $ d(x,p) < 1/H$ and extend by zero.
Then $||\phi||_{\infty} = ||\omega||_{\infty}$ and $||\phi||_2 \leq ||\omega||_2$.
The claim then holds for $C(H,\e) = 1/||\phi||_2$ by scaling the unit norm case.
\end{proof}
\begin{prop}\label{prop: 3.11} 
There is a constant $C = C(\e)$ such that if $f \in C^1(K)$ and $\omega = W_{\beta}(f)  = \alpha + \eta$ where $\alpha$ is $L^2$-coexact and $\eta$ is closed, then $$||\alpha||_{\infty} \leq C ||\alpha||_2.$$
\end{prop}

\begin{proof}
Assume that $||f||_2 = 1$. By Proposition \ref{prop: 3.9}, $||\nabla \alpha ||_{\infty} \leq H||f||_2 = H$.  Proposition \ref{prop: 3.10} gives a constant $C = C(H(\e))$ (so this really just depends on $\e$) such that $||\alpha||_{\infty}\leq C||\alpha||_2$. If $f$ does not have unit $L^2$-norm, then either $f=0$, in which case the result is trivial, or else $f = \lambda f’$ for some unit $L^2$-norm cochain $f’$ and positive number $\lambda$.
The coexact part of $W_{\beta}(f’)$ is $\alpha ‘ = \alpha/\lambda$. Hence, $||\alpha’||_{\infty}\leq C||\alpha’||_2$, and the result follows. \end{proof}

\section{The upper bound}
\label{sec:4}
In this section we prove Theorem A, which states that in a closed hyperbolic 3-manifold $M$ the first positive eigenvalue of the Hodge Laplacian acting on coexact 1-forms is bounded above by a multiple of the stable isoperimetric ratio $\rho(M)$. The background results of this section all hold in any dimension greater than 2, however the proof of Theorem A makes use of Poincar\'e duality to relate 1-forms and surfaces, this forces us to restrict Theorem A to the 3-dimensional setting.

The cochain results of the previous section are connected to spectral geometry via the inner product induced by the Whitney map associated to a triangulation and barycentric coordinate: $$\langle f, g\rangle = \int_M W_{\beta}(f)\wedge\star W_{\beta}(g),$$ which along with the corresponding norm $||\cdot||_2$, determine a Hodge theory for the cochain complex $C^{\bullet}(K)$. This inner product determines a codifferential $$d^*_W:C^{\bullet}(K)\to C^{\bullet-1}(K)$$ which, as the adjoint of the standard differential, satisfies $\langle d f, g\rangle = \langle f , d_W^* g \rangle.$
The corresponding Whitney Laplacian $\Delta_W:C^{\bullet}(K)\to C^{\bullet}(K)$ is then given by the standard formula $\Delta_W = dd_W^*+d_W^*d.$ This inner product was introduced using the standard barycentric coordinates in \cite{Dodziuk}.

This Laplacian decomposes the space $C^{\bullet}(K)$ into harmonic, exact, and coexact components: $$C^{\bullet}(K)\cong H^{\bullet}(M) \oplus dC^{\bullet-1}(K) \oplus d_W^* C^{\bullet+1}(K).$$ This combinatorial Hodge decomposition serves as a good approximation of the $L^2$-Hodge decomposition of $M$, though it does not capture the $L^2$-Hodge decomposition exactly. In particular, the Whitney coexact chains may not be $L^2$-coexact.

We begin by relating the Whitney and the Riemannian coexact eigenvalues.
\begin{lem} \label{lem:4.1}
Let $M$ be a closed Riemannian $n$-manifold with triangulation $K$ and an associated barycentric partition of unity $\beta$. Give the cochain complex the Whitney $L^2$-norm induced by the Whitney map determined by $\beta.$ Likewise, give the chain complex the dual norm $||\cdot||_2^*$ determined by the integration pairing. Then for every coexact cochain $f\in d_W^*C^2(K)$, there is an exact chain $a\in \d C_2(K)$ of unit norm such that $||f||_2 = \int_aW_{\beta}(f).$
\end{lem}
\begin{proof}

The cochain Hodge decomposition from the Whitney inner product gives the orthogonal decomposition $$C^1(K) = H^1(M) \oplus d_W^*C^2(K) \oplus dC_0(K).$$ Let $Z^1(K) = H^1(M)\oplus d C_0(K)$. Identify $C_1(K)$ with $C^1(K)^*$ via the integration pairing. The composition of the inclusion and quotient map determines an isomorphism $d_W^*C^2(K) \to C^1(K)/Z^1(K) $ that allows us to identify these spaces. If $\text{Ann}$ assigns to a subspace its annihalator, then there is also an isomorphism $(C^1(K)/Z^1(K))^*\to \text{Ann}(Z^1(K)).$ By Stokes’ theorem and dimension counting, $\text{Ann}(Z^1(K)) = \d C_2(K)$.
Thus, the dual of $d_W^*C^2(K)$ is exactly $\d C_2(K)$. The dual norm of an element $a \in \d C_2(K)$ is given by $$||a||_2^* = \sup\limits_{\substack{f\in C^1(K)\\||f||_2\leq 1}}\int_a W_{\beta}(f).$$
If $f$ has unit $L^2$-norm and $f = g + h$ where $g\in d_W^*C^2(K)$ and $h\in Z^1(K)$, then orthogonality implies $||g||_2\leq 1.$ Whence,
$$||a||_2^* = \sup\limits_{\substack{f = g + h\in C^1(K)\\||f||_2\leq 1}}\int_a W_{\beta}(g) = \sup\limits_{\substack{g\in d^*_WC^2(K)\\||g||_2\leq 1}}\int_a W_{\beta}(g).$$ The isometric identification of $(d^*_WC^2(K),||\cdot||_2)$ with its double dual therefore implies we can compute the norm of an element $f\in d^*_WC^2(K)$ via the integration pairing integrating only against chains in $\d C_2(K)$:
$$||f||_2= \sup\limits_{\substack{a \in \d C_2(K)\\||a||_2^* = 1}}\int_aW_{\beta}(f).$$ In particular, for any coexact cochain $f\in d^*_WC^1(K)$, there exists an exact chain $a$ with $||a||^*_2=1$ and $$\int_aW_{\beta}(f) = ||f||_2.$$
\end{proof}

\begin{prop} \label{prop:4.2}
Let $\lambda$ denote the first eigenvalue for the Hodge Laplacian acting on coexact $1$-forms and let $\lambda_W$ denote the first eigenvalue for the Whitney Laplacian acting on coexact 1-cochains associated to a deeply embedded triangulation $K$ and barycentric partition of unity $\beta$. There is a constant $G = G(\e)$ such that $$\lambda \leq G \vol(M)\lambda_W.$$
\end{prop}
\begin{proof}

The main issue here is that a Whitney coexact cochain will not generally map to an $L^2$-coexact form. This potentially adds a closed term to the denominator in the Whitney Rayleigh quotient, causing the Whitney Rayleigh quotient to be smaller than the Riemannian Rayleigh quotient. However, this failure can be controlled.

Let $f$ be a coexact eigen-cochain with eigenvalue $\lambda_W$. Set $\omega = W_{\beta}(f)\in\Omega^1(M)$, so that $\frac{||d\omega||_2^2}{||\omega||_2^2} = \lambda_W$. Let $p:\Omega^1(M)\to \Omega^1(M)$ be the $L^2$-orthogonal projection onto coexact forms. Let $a\in C_1(K)$ be the unit norm exact chain that realizes the norm of $f$ by integration given by Lemma 4.1. Then using that $d\omega = d(p(\omega))$ and the fact $a$ is exact, we obtain $$||f||_2 = ||\omega||_2 = \int_{a}\omega = \int_{a}p(\omega).$$

Using 3.11, we have $||p(\omega)||_{\infty}\leq C||f||_2.$ Hence, $||f||_2\leq C\text{len}(a)||p(\omega)||_2.$ We can therefore obtain a lower bound on $||p(\omega)||_2$ by bounding $\text{len}(a)$ from above. Applying Proposition \ref{prop: 3.1} to the chain $a$ above gives $$\mathcal ||a||_G\leq B \sqrt{\vol(M)} ||a||_2^* =  B \sqrt{\vol(M)}.$$ Since the lengths of the edges in the triangulation are bounded, we conclude the length of the support of $a$ is bounded.
Take $E$ to be the length of the largest edge possible in a deeply embedded triangulation, so that $\text{len}(a)\leq BE\sqrt{\vol(M)}$
Then we have obtained $$||\omega||_2 = \int_{a} \omega = \int_{a}p(\omega) \leq ||p(\omega)||_2  BCE\sqrt{\vol(M)}.$$ Setting $G = ( BCE)^2$ and using that $\omega$ is a Whitney eigenform, we obtain the result by the following short computation:
\[
 \lambda \leq \frac{||d\omega||_2^2}{||p(\omega)||_2^2}
 \leq G\vol(M)\frac{||d\omega||_2^2}{||\omega||_2^2}
 = G\vol(M)\lambda_W.
\]
\end{proof}

\begin{remark}
Note that the above estimate in fact holds for the first positive eigenvalue since the first positive eigenvalue $\lambda$ is the minimium of the first  eigenvalue of the Laplacian acting on functions and the first eigenvalue of the Laplacian acting on coexact 1-forms. The first eigenvalue $\lambda_f$ for the Laplacian acting on functions automatically satisfies the comparison $\lambda_f \leq \lambda_W$, as can be seen by studying the Rayleigh quotient and noticing that the estimate above controlling the projection in the denominator is immaterial in the function case.
\end{remark}

We are now ready to introduce stable commutator length, a thorough reference for which is \cite{Calegari}. For a group  $\G$, let  $\Gamma’$ denote the commutator subgroup and define the rational commutator subgroup to be $$\G_{\Q}’ = \text{Ker}(\Gamma\to \Gamma^{ab}\otimes \Q).$$ Note that when $\G$ is the fundamental group of a manifold, these subgroups correspond to the integrally nullhomologous and rationally nullhomologous loops respectively. The commutator length of an element $\gamma\in \Gamma’$, denoted ${\tt{cl}}(\gamma)$ is the shortest word length of $\gamma$ with respect to the generating set of all commutators.
The stable commutator length for $\gamma\in \G’_{\Q}$ is then defined to be \[\scl(\gamma) = \inf\limits_{m\geq 1}\frac{{\tt{cl}}(\gamma^m)}{m}.\] Topologically, stable commutator length corresponds to the stable complexity of a surface bounding a nullhomologous curve. In particular, for $\gamma\in \Gamma’_{\Q}$, one has $$\scl(\gamma) = \inf\left\{\frac{\chi_-(S)}{2m}~:~S \text{ with }\d S = \gamma^m \text{ and $S$ with no closed components}\right\},$$
where for a connected surface $S$ we define $\chi_-(S) = \max\{0,-\chi(S)\}$, and extend this additively to disconnected surfaces.
There is another natural complexity measure for loops in $\G’_\Q$, the Gersten filling norm. For a loop $\gamma\in \G’_\Q$, $\fill(\gamma)$ is the infimum of the Gromov norm $\frac{||A||_G}{m}$ for all singular 2-chains $A$ bounding a 1-cycle representing a singular fundamental class of $\gamma^m$. A fundamental theorem of Bavard relates the filling norm to the stable commutator length.

\begin{thm} \label{thm:4.3} (\cite{Bavard}) For any group element $\gamma$, there is an equality: \[\scl(\gamma) = 4\fill(\gamma).\]
\end{thm}
For proof, see for instance Lemma 2.69 in \cite{Calegari}.

\begin{remark} Let $B(\Gamma)$ be the $\R$-vector space of 1-boundaries. Then stable commutator length can be extended to a psuedo-norm on $B(\Gamma)$. After identifying chains with vanishing psuedo-norm, Bavard duality, which relates the filling norm to quasimorphisms and their defect norm, becomes a genuine functional analytic duality theorem. One could define the stable isoperimetric ratio in this chain setting, and the results of this paper would go through for that (smaller) ratio as well.
\end{remark}

We can now prove the main theorem of this section.

\begin{mainthm}\label{thm:A} Let $M$ be a closed hyperbolic 3-manifold  with injectivity radius bound below by $\e$.  There is a constant $A = A(\e)$ that only depends on $\e$ such that for any nontrivial boundary $\gamma \in \Gamma’_{\Q}$, one has
$$\sqrt{\lambda} \leq A \vol(M)\frac{|\gamma|}{\scl(\gamma)},$$
 where $\lambda$ is the first coexact eigenvalue of the Hodge Laplacian on $\Omega^1(M)$.
\end{mainthm}

\begin{proof}
First note that since stable commutator length and geodesic length are both multiplicative under powers, it suffices to show the claim for an integrally nullhomologous loop $\gamma$.

Fix a deeply embedded triangulation $K$ of $M$ and denote by $\lambda_W$ the first eigenvalue of the Whitney Laplacian $\Delta_W$ acting on $d^*_WC^2(K)$ associated to a smooth barycentric partition of unity. Notice that the Hodge decomposition ensures that zero is not an eigenvalue of this operator. Let $c:S^1\to M$ be a cellular path in the 1-skeleton of $K^*$ representing the loop $\gamma$, constructed as in Proposition  \ref{prop: length comparison}.
Let $T$ be a triangulation of $K^*$. Let $a\in C_1(K^*)$ be the fundamental cycle for $\gamma$ corresponding to the path $c$ in $C_1(K^*)\subset C_1(T)\subset C_1^{\text{sing}}(M)$.
If $\Phi:C^2(K)\to C_1(K^*)$ is the Poincar\'e duality map, then $\Phi^{-1}(a)$ is an exact 2-cochain. We can therefore choose $\omega\in d^*_WC^2(K)$ with $d\omega = \Phi^{-1}(a)$. Setting $A = \Phi(\omega)$ in $C_2(K^*)$, we have $\d A = a$ and $||A||_G = ||\omega||_G.$
Since $\lambda_W$ is nonzero, we have $||\omega||_2 \leq \frac{||d\omega||_2}{\sqrt{\lambda_W}}$.  Proposition \ref{prop:4.2} implies  \[||\omega||_2\leq \frac{\sqrt G\sqrt{\vol(M)} ||d\omega||_2}{\sqrt{\lambda}}.\] By Bavard’s theorem relating the filling norm to stable commutator length (Theorem 4.3), our choice of $A$, and Proposition \ref{prop: 3.4} we find that
$$\scl(\gamma) = 4\fill(\gamma) \leq 4||\tau(A)||_G \leq 4N||A||_G = 4N||\omega||_G,$$ where, as in Proposition \ref{prop: 3.4}, $\tau$ is the triangulation map relating the cellular chain $A$ to the subdivided simplicial chain in $C_2(T)$. Consequently,
\begin{align*}
                  \scl(\gamma) &\leq 4N ||\omega||_G\\
                 &\leq 4NB\sqrt{\vol(M)}||\omega||_2 ~\text{by Proposition \ref{prop: 3.1},}\\
                 &\leq  4NB\sqrt G \vol(M)\frac{||d\omega||_2}{\sqrt{\lambda}} ~\text{by above computation,}\\
                 &\leq 4NB\sqrt G \vol(M)\frac{D||d\omega||_G}{\sqrt{\lambda}}~\text{by Proposition  \ref{prop: 3.2},} \\
                 &=4NB\sqrt G \vol(M)\frac{D||\d  A||_G}{\sqrt{\lambda}}~\text{by Proposition  \ref{prop: 3.3},} \\
                 &=  4NB\sqrt G \vol(M)\frac{D||c||_G}{\sqrt{\lambda}} \text{ by construction of $\d A$,} \\
                 &\leq 4NB\sqrt G \vol(M)\frac{DL|\gamma|}{\sqrt{\lambda}}~\text{by Proposition  \ref{prop: length comparison},} \\
                 &= 4NB\sqrt G DL\vol(M)\frac{|\gamma|}{\sqrt{\lambda}}.
\end{align*}
Setting $A=4NB\sqrt G DL$ and rearranging, we are done.
\end{proof}

\section{The lower bound}
\label{sec:5}

We now turn to proving the lower bound on the first coexact eigenvalue of the 1-form Laplacian that constitutes Theorem B. Unlike Theorem A, we prove this eigenvalue comparison without a dimension constraint. The line of proof follows that of Theorem 1.3 in \cite{LS}.

In \cite{LS}, the authors obtain the following estimate controlling the $L^2$-norm of coclosed forms.  Note that this estimate does not depend on the fundamental domain coming from a deeply embedded triangulation.

\begin{prop} \label{prop:5.1} (Proposition 5.4 in \cite{LS})
Let $\eta$ be a 1-form on $M$ and $\mathcal D\subset \H^n$ any fundamental domain. Then,
$$||\eta||_2^2 \leq \emph{Area}(\d \mathcal D)||\eta||_{\infty}\left(3\pi||d\eta||_{\infty} + \max_i\left|\int_{\gamma_i}\eta\right|\right) + \frac{1}{2}||d\eta||_{\infty}||\eta||_2\sqrt{\vol(M)},$$
where the $\gamma_i$ are the geodesics in the homotopy class of the loops representing the side pairing transformations of the fundamental domain $\mathcal D$.
\end{prop}

Studying the terms in the estimate of Proposition \ref{prop:5.1} for a coexact $\lambda$-eigenform provides a lower bound on $\lambda$ given later as Theorem B. The essential idea is that after applying an $L^2$-$L^{\infty}$ norm comparison, all but one summand on the right-hand side (the integral term), has a $||d\eta||_2$ term. In particular, if $\eta$ is a unit norm eigenform, the right-hand side almost has a $\sqrt{\lambda}$ term in every summand. Our aim, then, is to replace the integral term with something that looks like $||d\eta||_{\infty} (\rho(M)^{-1}+\text{stuff})$, where the stuff is polynomial in the volume of $M$ with constants that depend only on the lower bound on injectivity radius.

\begin{lem} \label{lem:5.2}Let $a$ be the lift to $\H^n$ of the cellular approximation of a geodesic loop and let $\gamma$ be the (oriented) lift of the same geodesic loop. If $\tilde \eta$ is the pullback to $\H^n$ of a 1-form $\eta\in \Omega^1(M)$, then $$\left|\int_{a}\tilde\eta - \int_{\gamma}\tilde \eta\right| \leq \pi||d\eta||_{\infty}||a||_G.$$ \end{lem}
\begin{proof}
Let $x$ be the starting point of $\gamma$, let $a_i$ be the geodesic arcs of $a$ (so that the word corresponding to the cellular path $a$ is the word $a_1\cdots a_{||a||_G}$), and let $y$ be the end point of $a_{||a||_G}$. Let $Q$ be the piecewise geodesic $(||a||_G+3)$-gon obtained by taking the union of the triangles $\text{convex hull}(a_i,x)$ and the triangle $\text{convex hull}(\gamma,y)$. Since $Q$ is the union of $(||a||_G + 1)$ geodesic triangles of area bounded by $\pi$, we have the upper bound of $(||a||_G+1)\pi$ for the area of $Q$.
Then, since  $$
\left|\int_a\tilde\eta - \int_{\gamma}\tilde\eta\right| = \left|\int_{\d Q} \tilde\eta\right|  \leq \pi||d\eta||_{\infty}(||a||_G+1),$$ where the first equality follows from the deck transformation invariance of $\tilde \eta$, which makes the integral over $\d Q\setminus(a\cup\gamma)$ vanish.
If one then considers the cellular
s $ma$ and the geodesic $\gamma^m$ integrated over the form $\frac{1}{m}\tilde\eta$, one gets
$$ \left|\int_a\tilde\eta - \int_{\gamma}\tilde\eta\right| = \frac{1}{m}\left|\int_{ma}\tilde\eta - \int_{\gamma^m}\tilde \eta\right|\leq \frac{1}{m}\pi||d\eta||_{\infty}(||ma||_G+1).$$
Taking the limit as $m\to\infty$ then results in  $$\left|\int_{a}\tilde\eta - \int_{\gamma} \tilde\eta\right|\leq\pi||d\eta||_{\infty}||a||_G,$$ proving the lemma.
\end{proof}

\begin{lem}\label{lem:5.3}
There is a constant $B_0 = B_0(\e)$ such that $\diam(M)\leq B_0\vol(M)$.
\end{lem}
\begin{proof}
	First note that there is a constant $T = T(\e)$ such that the number of simplices in $M$ is bounded by $T\vol(M)$ and that each simplex from a deeply embedded triangulation has bounded diameter, say bounded by $C_0$. With $B_0 = C_0T$, one has that $B_0\vol(M)$ bounds the diameter of $M$, as desired.
\end{proof}

\begin{lem}\label{lem:5.4}
Let $\eta\in\Omega^1(M)$ be a 1-form and $\gamma$ a rationally nullhomologous loop in $M$. Then integrating over the geodesic in the free homotopy class satisfies $$\left|\int_{\gamma}\eta \right|\leq 2 \pi||d\eta||_{\infty}\scl(\gamma).$$
\end{lem}
\begin{proof}

This follows from Bavard duality and the fact that $\int_{\gamma}\eta$, where the integral is over the geodesic in the free homotopy class of $\gamma$, is a quasimorphism with defect bounded by $\pi||d\eta||_{\infty}$ (see \cite{Calegari}, page 21).
\end{proof}

The key estimate allowing us to replace the integral term with one involving the stable isoperimetric constant $\rho(M)$ is the following proposition; compare with Proposition 5.24 in \cite{LS}.

\begin{prop} \label{prop:5.5} Let $\eta$ be a 1-form on $M$. Then there is a harmonic form $h$ and a constant $L_0 = L_0(\e)>0$ such that for every closed geodesic $\alpha$ in $M$, one has $$\left|\int_{\alpha}(\eta - h)\right|\leq |\alpha| L_0\vol(M)^{3/2} ||d\eta||_{\infty}\left(\rho(M)^{-1} + 1 \right).$$
\end{prop}
\begin{proof}

If $M$ is a $\Q$-homology sphere, $\alpha$ is rationally nullhomologous and $h$ can only be 0. Lemma \ref{lem:5.4} gives $$\left|\int_{\alpha}\eta\right|\leq 2\pi\scl(\alpha)||d\eta||_{\infty}.$$
By multiplying the right-hand side by $|\alpha|/|\alpha|$, this becomes $$ \left|\int_{\alpha}\eta\right|\leq 2\pi |\alpha|\frac{\scl(\alpha)}{|\alpha|}||d\eta||_{\infty}\leq 2\pi|\alpha|\rho(M)^{-1}||d\eta||_{\infty}.$$
If $W$ is the minimal volume hyperbolic $n$-manifold (in dimension 3, this is the Weeks manifold, see\cite{minvol}, more generally it is know that in dimension $n\geq4$ that the set of hyperbolic volumes is discrete in $\R$, see \cite{minvol2}), the claim follows with $L_0 = \frac{2\pi}{\vol(W)^{3/2}}$.

Thus, we assume $M$ has nontrivial real homology classes. Fix a basepoint $x_0\in M$. Take a basis $c_1,\dots, c_n$ of harmonic 1-chains for  $C_1(K^*)$ using the Euclidean inner product on $C_1(K^*)$. Harmonic chains are norm minimizing for the induced $\ell^2$-norm. Let $||\cdot||_E$ denote this norm. Let $h$ be the (unique) harmonic form that satisfies $\int_{c_i}\eta - h = 0$ for each $i$. Let $a$ be a cellular path in $K^*$ approximating $\alpha$, as in Proposition \ref{prop: length comparison}, so $||a||_G \leq L|\alpha|$ and identify the cellular path $a$ with the chain it represents.

Then, using the Hodge decomposition induced by the Euclidean inner product, we get $a = a_h^E + \d S,$ where $a_h^E$ is harmonic with respect to the Euclidean inner product on the chain complex $C_1(K^*)$ and $S$ is some 2-chain. Since $\d$ and the Euclidean adjoint $\d^*_E$ have integral bases, and since $a$ is integral, $\d S$ is a rational 2-chain. Recall that there is a constant $T=T(\e)$ such that the number of 2-simplices in $K^*$ is bounded by $T\vol(M)$. A short computation then shows,
\begin{align*}
||\d S||_G &= ||a - a_h^E||_G\\
		&\leq ||a||_G+||a_h^E||_G\\
		&\leq ||a||_G+\sqrt{T\vol(M)}||a_h^E||_E, \text{~by the Euclidean $\ell^1$-$\ell^2$ norm comparison,}\\
		&\leq ||a||_G + \sqrt{T\vol(M)}||a||_E, \text{~ as $a_h^E$ is $\ell^2$-norm minimizing in its class},\\
		&\leq ||a||_G + \sqrt{T\vol(M)}||a||_G, \text{~by the Euclidean $\ell^1$-$\ell^2$ norm comparison},\\
		&= (\sqrt{T\vol(M)}+1 )||a||_G.
\end{align*}
Because there is a minimal volume hyperbolic 3-manifold, we can increase $T$ so that we can write the above as $||\d S||_G\leq T\sqrt{\vol(M)}||a||_G$. Additionally, since there is a universal upper bound on the length of an edge in $K^*$, there is a constant $E>0$, such that the geodesic length $|\gamma|$ of a loop $\gamma$ satisfies $|\gamma|\leq E\len(c)$ for any cellular path $c$ in $K^*$ homotopic to $\gamma$.

Since $\d S$ is a rational cycle, take $N>0$ to be an integer so that $N\d S$ is integral. Then one can glue together oriented copies of the edges on which $\d S$ is supported along their boundaries to obtain a (non unique) collection of closed cellular loops $b_1,\dots,b_m$ whose union represents the cycle $N\d S$.  Fix a vertex $v_i$ in each loop $b_i$. Note that by construction, $\len(b_i) = ||b_i||_G$ for each $i$. Let $\tau_i$ be the geodesic arc connecting the basepoint $x_0$ to $v_i$ and $\tau_i^{-1}$ the oppositely oriented geodesic arc. Define the curve $b$ to be the path $$\tau_1b_1\tau_1^{-1}\tau_2b_2\tau_2^{-1}\cdots \tau_mb_m\tau_m^{-1}.$$
Let $\beta$ be the geodesic loop through $x_0$ homotopic to $b$. Notice $\sum_i ||b_i||_G = ||b||_G$, where $||b||_G$ is meant in the sense of the norm on singular chains, where the $\tau^{\pm 1}$ terms cancel. This gives a possibly trivial element of $\Gamma_\Q’$ whose length is bounded as follows:

\begin{align*}|\beta| \leq |b| &= \sum_i (2|\tau_i| + |b_i|) \\ &\leq 2\diam(M)m + E||b||_G \\ &\leq (2\diam(M) + E)||b||_G \\ &\leq (2B_0\vol(M) + E)||b||_G,
\end{align*}
where we use that $m \leq ||b||_G$, the diameter bound of Lemma \ref{lem:5.3}, along with the remarks in the above discussion.

Since $$\frac{1}{N}||b||_G = ||\d S||_G \leq T\sqrt{\vol(M)}||a||_G,$$ and  $||a||_G\leq L|\alpha|$, we obtain $$\frac{||b||_G}{N} \leq TL\sqrt{\vol(M)}|\alpha|.$$ As a result, $$\frac{|\beta|}{N}\leq TL(2B_0\vol(M)+E)\sqrt{\vol(M)}|\alpha|.$$
We compute,
\begin{align*}
\left|\int_{\alpha}\eta-h\right| &= \left|\int_{\alpha-a_h} \eta-h \right|\text{~since $a_h$ is in the span of the $c_i$, and $\int_{c_i}\eta-h$ = 0,}\\
&\leq\left| \int_{\d S}\eta- h\right| + \left|\left(\int_{\alpha} \eta - h\right) - \left(\int_{a}\eta-h\right)\right| \\
&\leq\left| \int_{\d S}\eta- h\right| + \pi||d\eta||_{\infty}||a||_G,\text{~by Lemma \ref{lem:5.2},} \\
&= \frac{1}{N}\left|\int_{N\d S}\eta-h\right| + \pi||d\eta||_{\infty}||a||_G\\
&= \frac{1}{N}\left|\int_{b}\eta-h\right| + \pi||d\eta||_{\infty}||a||_G,\text{~since $b$ abelianizes to $N\d S$,}\\
&\leq \frac{1}{N}\left(\left|\int_{\beta}\eta-h\right| + \left|\int_b(\eta-h) - \int_{\beta}(\eta-h)\right|\right) + \pi||d\eta||_{\infty}||a||_G\\
&\leq \frac{1}{N}\left|\int_{\beta}\eta-h\right| + \frac{1}{N}\pi||d\eta||_{\infty}||b||_G + \pi||d\eta||_{\infty}||a||_G,\text{~by Lemma \ref{lem:5.2}}.
\end{align*}

If $\beta$ is trivial, then the integral term $|\int_{\beta}\eta-h| $ vanishes, and we can replace that term with $$TL\pi |\alpha| ||d\eta||_{\infty}\sqrt{\vol(M)}\rho(M)^{-1} $$ to obtain (after using our estimate for $||b||_G/N$ and $||a||_G \leq L|\alpha|$)

\begin{align*}
\left|\int_{\alpha}\eta-h\right| &\leq TL\pi|\alpha|||d\eta||_{\infty}\sqrt{\vol(M)}\left(\rho(M)^{-1}+1\right) + \pi L ||d\eta||_{\infty}|\alpha|\\
&\leq TL\pi|\alpha|||d\eta||_{\infty}\sqrt{\vol(M)}\left(\rho(M)^{-1}+1\right) + \frac{\vol(M)^{3/2}}{\vol(W)^{3/2}}\pi L ||d\eta||_{\infty}||\alpha|\\
&\leq TL\pi|\alpha|||d\eta||_{\infty}\frac{\vol(M)^{3/2}}{\vol(W)}\left(\rho(M)^{-1}+1\right) + \frac{\vol(M)^{3/2}}{\vol(W)^{3/2}}\pi L ||d\eta||_{\infty}|\alpha|,
\end{align*}
where in the last line we again use the minimal volume hyperbolic 3-manifold $W$ to replace $\sqrt{\vol(M)}$ with $\vol(M)^{3/2}.$
Setting $$L_0 = 2\max\{\frac{2\pi BDL}{\vol(W)},\frac{\pi L}{\vol(W)^{3/2}}\}$$ and factoring gives the result.

Assume now that $\beta$ is nontrivial. Combining the above estimates yields

\begin{align*}
\left|\int_{\alpha}\eta-h\right| &\leq \frac{|\beta|}{N} \frac{1}{|\beta|} \left|\int_{\beta}\eta-h\right| + \frac{1}{N}\pi||d\eta||_{\infty}||b||_G + \pi L ||d\eta||_{\infty}|\alpha| \\
&\leq TL(2B_0\vol(M)+E)\sqrt{\vol(M)}|\alpha| \frac{1}{|\beta|} \left|\int_{\beta}\eta-h\right| \\
&~~~~+ \pi||d\eta||_{\infty} TL\sqrt{\vol(M)}|\alpha | + \pi L ||d\eta||_{\infty}|\alpha|
 \\
&= TL\sqrt{\vol(M)}|\alpha|\left((2B_0\vol(M) + E)\frac{1}{|\beta|}\left|\int_{\beta}\eta-h\right|  + \pi||d\eta||_{\infty}\right) \\
&~~~~+ \pi L ||d\eta||_{\infty}|\alpha|.
\end{align*}

Since the geodesic $\beta$ is nullhomologous, Lemma \ref{lem:5.4} implies $$\left|\int_{\beta}(\eta-h) \right|\leq 2\pi||d\eta||_{\infty}\scl(\beta).$$ Replacing the integral term with this estimate and using that $\frac{\scl(\beta)}{|\beta|}\leq\rho(M)^{-1}$ gives

\begin{align*}
\left|\int_{\alpha}\eta-h\right| \leq |\alpha| TL\sqrt{\vol(M)} ||d\eta||_{\infty}\left(2\pi(B_0\vol(M) + E)\rho(M)^{-1}\right)+ \pi) \\
+ \pi L ||d \eta||_{\infty}|\alpha|.
\end{align*}

Again using the existence of a minimal volume hyperbolic $n$-manifold, one can replace $B_0$ with the constant $B_1 = 2B_0 + E/\vol(W)$ since $B_1\vol(M) > 2B_0\vol(M) + E$. Then, after combining constants in the first summand (and using that $2\pi > \pi$ to pull out the terms containing $\pi$) into a single constant $L_1$, one obtains:

$$\left|\int_{\alpha}(\eta - h)\right|\leq |\alpha| L_1\vol(M)^{3/2} ||d\eta||_{\infty}\left(\rho(M)^{-1} + 1 \right) + \pi L_1 ||d\eta||_{\infty}|\alpha|. $$

Set $L_0 =2\max\{L_1,\frac{\pi L}{\vol(W)^{3/2}}\}$ and multiply the second summand by $\vol(M)^{3/2}$ to obtain the claim.

\end{proof}

\begin{lem}  \label{lem:5.6} Let $M$ have deeply embedded triangulation $K$ and let $\tilde K$ be the pullback of this triangulation to $\H^n$.
Then there is a fundamental domain $\mathcal D\subset \H^n$ for $M$ that is a union of simplices from $\tilde K$ such that the diameter of $\mathcal D$ satisfies $\diam(\mathcal D) \leq 3\diam(M).$ \end{lem}
\begin{proof}
Fix a top dimensional simplex $\sigma_0\in K^{(n)}$ and let $\tilde \sigma_0$ be a lifted copy in $\H^n$. Let $\tilde x_0$ be the barycenter of $\tilde \sigma_0$. For every other top dimensional simplex $\sigma$ in $K^{(n)}$ there is a lift 	$\tilde \sigma$ whose barycenter $\tilde x_{\sigma}$ is within $\diam(M)$ of $\tilde x_0$. Choose one such lift for every $\sigma$ in such a way the resulting fundamental domain $\mathcal D$ is connected. Then the diameter of the fundamental domain satisfies $\diam(\mathcal D)\leq \diam(M) + 2e$, where $e$ is the maximum distance from the barycenter of a simplex in a deeply embedded triangulation to its boundary. Clearly $e<\diam(M)$, so the lemma immediately follows.
\end{proof}

Now, assume $M$ has a fixed deeply embedded triangulation and let $\mathcal D$ be a fundamental domain as in Lemma \ref{lem:5.6}. Let $\gamma_i$ be the geodesics in the free homotopy class of the side pairing transformations of the fundamental domain $\mathcal D$, and notice by construction $|\gamma_i|\leq 3\diam(M)$. With this, we modify the estimate in Proposition \ref{prop:5.1}to obtain the following.

\begin{prop}\label{prop:5.7}
Let $\eta$ be a coclosed 1-form on $M$. Let $h$ be the harmonic form of Proposition \ref{prop:5.5} associated to $\eta$. Then for a constant $A_0 = A_0(\e)>0$, the following holds:
\begin{align*}
||\eta-h||_2^2 \leq A_0\vol(M)||\eta-h||_{\infty}\left(3\pi||d\eta||_{\infty} + 3L_0B_0\vol(M)^{5/2} ||d\eta||_{\infty} \left(\rho(M)^{-1} + 1\right)\right) \\ + \frac{1}{2}||d\eta||_{\infty}||\eta-h||_2\sqrt{\vol(M)}.
\end{align*}

\end{prop}
\begin{proof}

Let $\gamma_i$ realize the maximum among the integrals $\int_{\gamma_i}\eta$.
Substitute the estimate of Proposition \ref{prop:5.5} for the integral term in Proposition \ref{prop:5.1} applied to the coclosed form $\eta-h$ and the fundamental domain $\mathcal D$ to obtain
\begin{align*}
    ||\eta-h||_2^2 \leq \text{Area}(\d \mathcal D)||\eta-h||_{\infty}
    \left(3\pi||d\eta||_{\infty} +L_0\vol(M)^{3/2} ||d\eta||_{\infty} |\gamma_i| \left(\rho(M)^{-1} +1 \right)\right) \\ + \frac{1}{2}||d\eta||_{\infty}||\eta-h||_2\sqrt{\vol(M)}.
\end{align*}

Then, replace $|\gamma_i|$ with $3B_0\vol(M)$, using Lemma \ref{lem:5.6} and Lemma \ref{lem:5.3}.
Lastly, since there is an upper bound on the area of a face of any simplex in $\mathcal G_\e$ (a consequence of the bounds on the dihedral angles), the total area of the boundary of a complex made from no more than $T\vol(M)$ simplices from $\mathcal G_\e$ is bounded by $A_0\vol(M)$ for a constant $A_0$ depending on $\e$. Substituting this estimate for the $\text{Area}(\d \mathcal D)$ term completes the proof.
\end{proof}

\begin{prop}\label{prop:5.8} (Proposition 2.2 of \cite{LS}) Let $M$ be a closed hyperbolic n-manifold with $\inj(M)>\e$. Assume the first positive eigenvalue $\lambda$ of the Laplacian acting on coexact 1-forms is less than some fixed constant $H>0$. Then there is a constant $C(H,\e)>0$ such that for a coexact $\lambda$-eigenform $\omega$, one has $$||\omega||_{\infty}\leq C(H,\e)||\omega||_2.$$
\end{prop}

\begin{prop}\label{prop:5.9}
Let $M$ be a closed hyperbolic n-manifold with $\inj(M)> \e$. Let $\lambda < H$ be the first positive eigenvalue for the Hodge Laplacian acting on coexact 1-cochains. Then the following holds:
\begin{align*}\frac{1}{\sqrt{\lambda}} \leq A_0\vol(M)C(H,\e)^2
    \left(3\pi+ 3L_0B_0\vol(M)^{5/2}\left(\rho(M)^{-1} + 1\right)\right) \\
    + \frac{C(H,\e)}{2}\sqrt{\vol(M)}.
\end{align*}
\end{prop}

\begin{proof}Let $\eta$ be a $\lambda$ coexact eigenform.
Applying the Sobolev type estimate of Proposition \ref{prop:5.8} to each instance of the sup norm in Proposition \ref{prop:5.7} and using that $||d\eta||_2=\sqrt{\lambda}||\eta||_2 \leq \sqrt{\lambda}||\eta - h||_2$, where the inequality follows from the orthogonality of the Hodge decomposition, gives

\begin{align*}||\eta-h||_2^2 \leq A_0\vol(M) C(H,\e) ^2 \sqrt{\lambda} ||\eta-h||_{2}^2 \left(3\pi + 3L_0B_0\vol(M)^{5/2}\left(\rho(M)^{-1}
    + 1\right)\right) \\ + \frac{C(H,\e)}{2}\sqrt{\lambda} ||\eta-h||_2^2\sqrt{\vol(M)}.
\end{align*}

Dividing both sides by $\sqrt{\lambda} ||\eta-h||_2^2$ then gives \begin{align*} \frac{1}{\sqrt{\lambda}}\leq A_0\vol(M) C(H,\e) ^2\left(3\pi+ 3L_0B_0\vol(M)^{5/2}\left(\rho(M)^{-1} + 1\right)\right) \\ + \frac{C(H,\e)}{2}\sqrt{\vol(M)}.\end{align*}\end{proof}

Rearanging the terms and combining constants (which again requires the existence of a minimal volume hyperbolic $n$-manifold) in the previous proposition and applying a geometric estimate of Calegari and a systolic inequality due to Sabourau leads to the main theorem of this section.

\begin{mainthm}\label{thm:B}
Let $M$ be a closed hyperbolic $n$-manifold with $\inj(M) > \e$. Let $\lambda$ be the first positive eigenvalue for the Laplacian acting on coexact 1-forms and let $H > \lambda$. Then there is a constant $P(H,\e)>0$ such that $$ \frac{P\rho(M)}{\vol(M)^{7/2+1/n}}\leq\sqrt{\lambda}.$$
\end{mainthm}
\begin{proof}
First we rearrange the previous proposition and combine constants into one constant $P$ to get the estimate $$\frac{P\rho(M)}{(1+\rho(M))\vol(M)^{7/2}}\leq\sqrt{\lambda}.$$ We need an estimate of Calegari’s (see the proof of Theorem 3.9 in \cite{Calegari}, the estimate at the bottom of page 58) which gives that for a genus $g$ surface $S$ with boundary $\d S = \gamma^m$,
one has $$\frac{m|\gamma|}{12g-6}\leq 4\mu + \frac{2\pi}{3\mu} + 2|\gamma|,$$ where $\mu$ depends only on the dimension $n$. Since $\chi_-(S)\geq2g-1,$ we get $$\frac{2m|\gamma|}{\chi_-(S)}\leq 24\left(4\mu + \frac{2\pi}{3\mu} + 2|\gamma|\right).$$
Since this is true for any surface $S$ bounding a power of $\gamma$, we obtain $$\frac{|\gamma|}{\scl(\gamma)}\leq 24\left(4\mu + \frac{2\pi}{3\mu} + 2|\gamma|\right).$$
We also have the commutator systolic inequality of Sabourau from Theorem 1.4 in \cite{Sab}, which bounds the shortest nontrivial integrally nullhomologous loop $\gamma\in \Gamma’$ by $$|\gamma|\leq c\vol(M)^{1/n},$$ for a dimensional constant $c$.

Both the inequality of Calegari and the systolic inequality involve a dimensional constant; let $\mu$ be the maximum of these constants in dimension $n$ and write Calegari’s inequality as $\frac{|\gamma|}{\scl(\gamma)} \leq \mu(1+|\gamma|)$. Then we get $$ \rho(M)\leq \frac{|\gamma|}{\scl(\gamma)}\leq \mu(1 + |\gamma|) \leq \mu (1+ \mu\vol(M)^{1/n}).$$
Inserting this upper bound into the denominator of the above rearranged estimate above gives $$\frac{P\rho(M)}{(1 + \mu(1+\mu\vol(M)^{1/n})\vol(M)^{7/2}}\leq \frac{P\rho(M)}{(1+\rho(M))\vol(M)^{7/2}} \leq \sqrt{\lambda}.$$ We can then increase $P$ to allow us to pull out the volume term and absorb $\mu$, thereby obtaining the desired estimate.
\end{proof}

\section{An example}
\label{sec:6}
\begin{figure}[H]

\includegraphics[scale=.6]{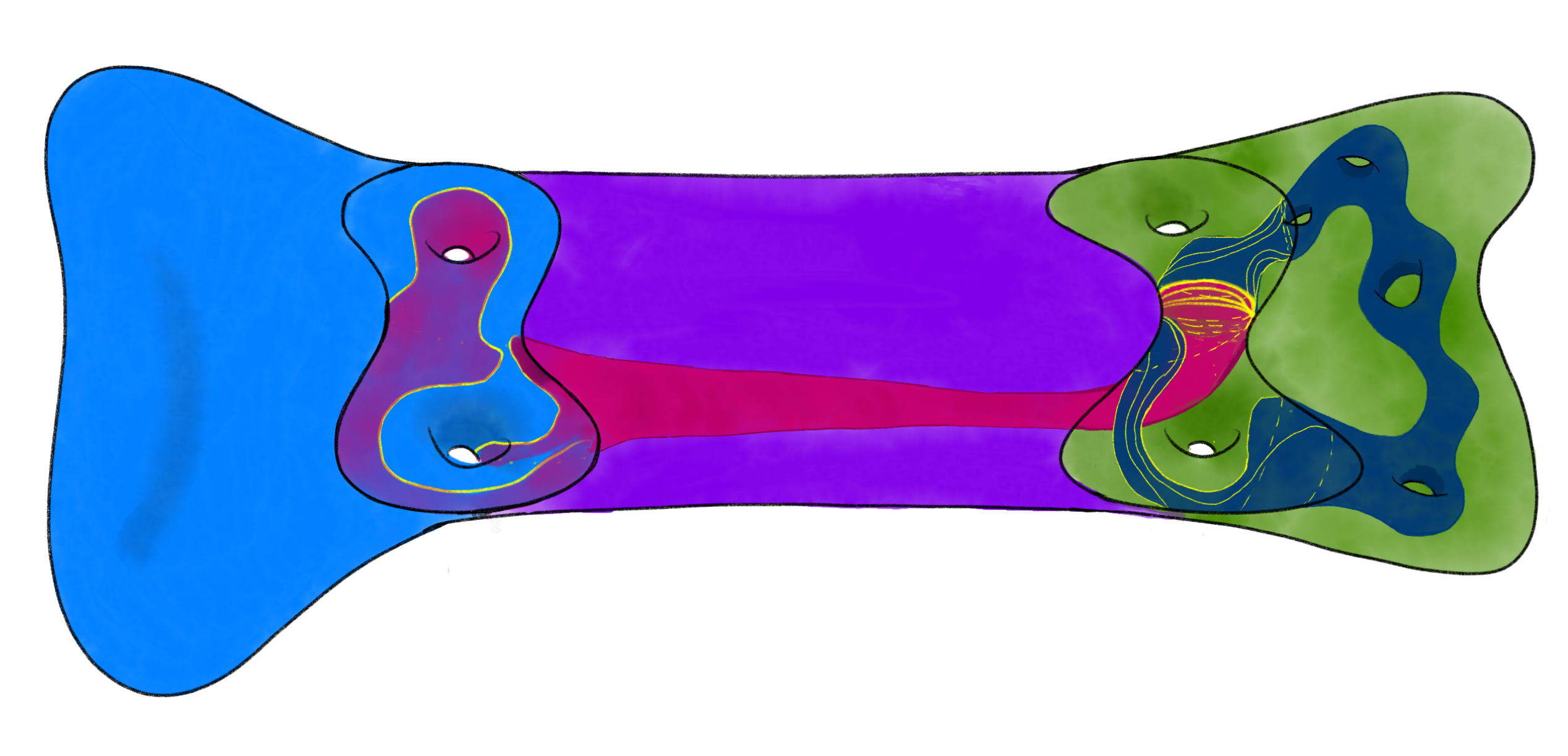}
\centering
\caption{ A sketch of the example constructed in this section. The manifold decomposes into three pieces, left and right caps and a middle region. The loop drawn on the boundary of the left cap only bounds surfaces of high topological complexity that are at least partly contained in the right cap. This ensures the isoperimetric ratio of that loop is very small.}
\label{fig:example_sketch}
\end{figure}

The aim of this section is to show that the first positive eigenvalue of the 1-form Laplacian can vanish exponentially fast in relation to volume. This contrasts the behaviour of the first positive eigenvalue of the Laplacian on functions.

Our construction is similar to that in \cite{BD}. Essentially, we choose a hyperbolic 3-manifold with totally geodesic boundary and glue it to itself using a particular psuedoAnosov with several useful properties. By \cite{BMNS}, this family has geometry that up to bounded error can be understood in terms of a simple model family. Using this model family, we show that one can find curves with uniformly bounded length whose stable commutator length grows exponentially in the volume. We then use the spectral gap upper bound in Theorem A to conclude the first positive eigenvalue vanishes exponentially fast.

Throughout this section, we need to compare geodesic lengths in different submanifolds of a given manifold $M$. Let $|\cdot|_{X}$ denote the geodesic length of a homotopy class of curves relative endpoints in a manifold $X$ and $\text{length}(\cdot)$ be the length in $M$ of the curve. Similarly, when we compute stable commutator length for the fundamental group of a manifold $X$, which may or may not be a submanifold of $M$, we denote it $\scl_X$.

We will need that for certain curves, $\scl$ is comparable to length. We begin with a simple but essential technical lemma.
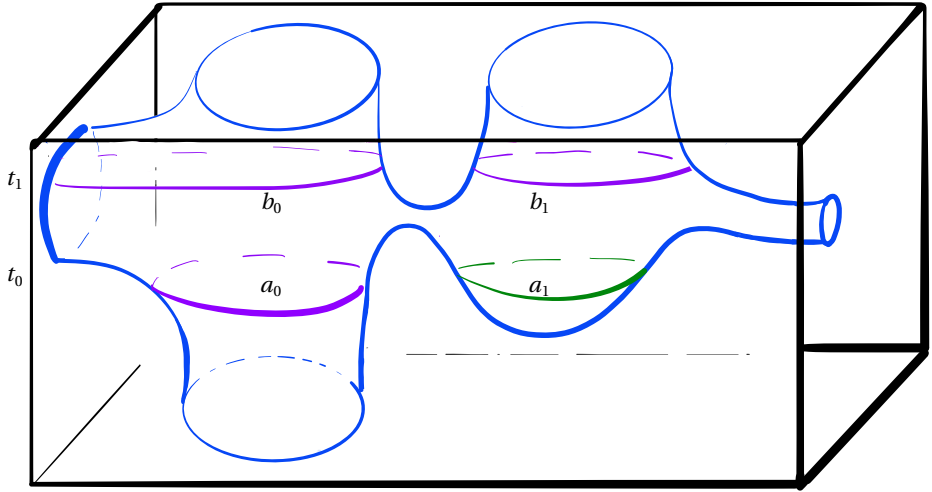
\begin{figure}[H]
\labellist
\small\hair 2pt
 \pinlabel {$a_0$} [ ] at 830 1100
 \pinlabel {$a_1$} [ ] at 1300 1100
 \pinlabel {$b_0$} [ ] at 830 1250
 \pinlabel {$b_1$} [ ] at 1300 1250
 \pinlabel {$t_1$} [ ] at 380 1290
 \pinlabel {$t_0$} [ ] at 380 1120
\endlabellist
\centering
\includegraphics[width = 15cm]{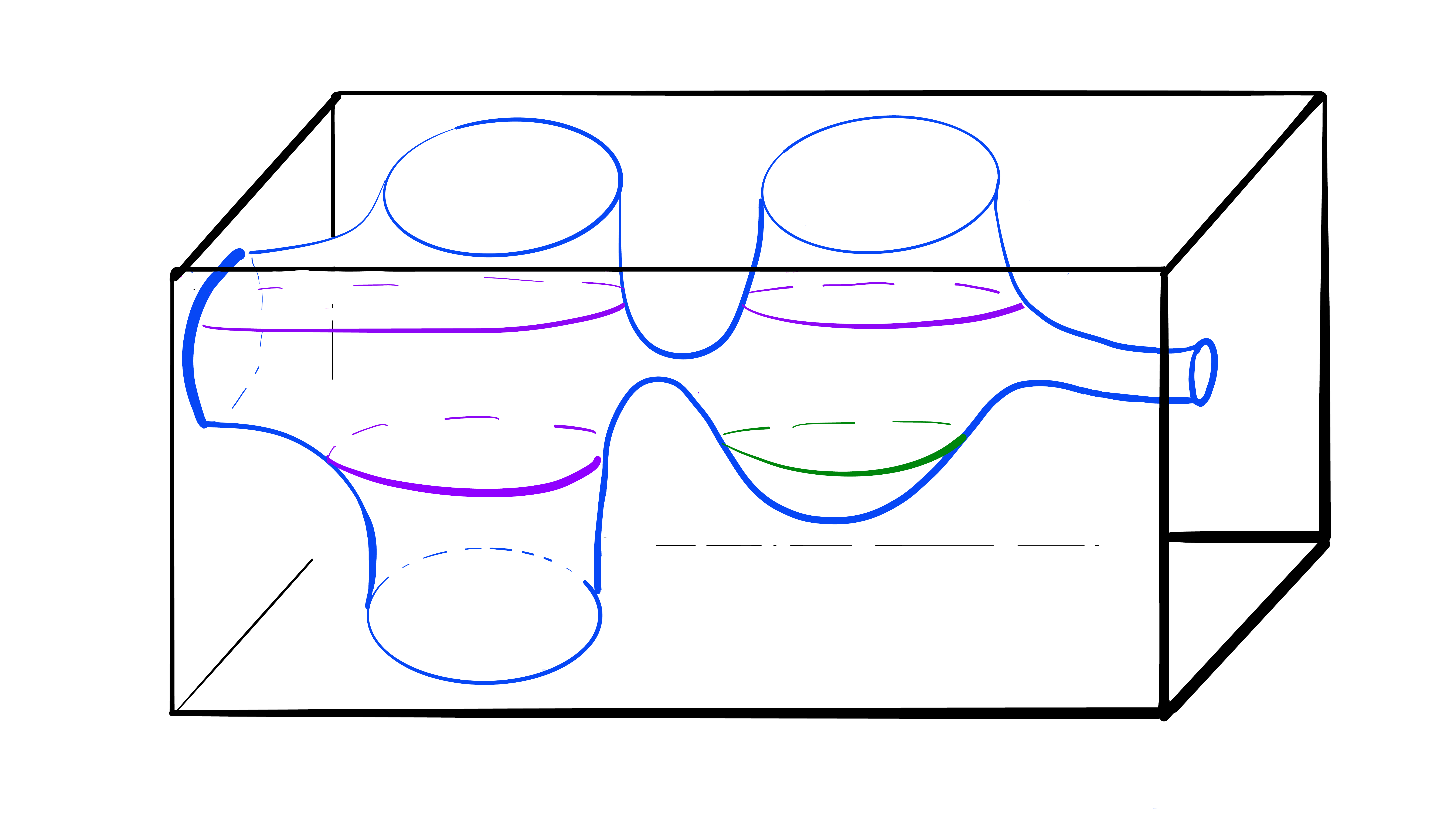}
\caption{ Illustrating Lemma \ref{lem:6.1} , the rectangular base of the figure is part of the totally geodesic surface $S$ and the box is the corresponding part of the tubular neighborhood $N_\e(S)$ foliated by surfaces $S_t$ parallel to $S$. Drawn in the box is the surface $\Sigma$, which is transverse the foliation except at isolated points. The multicurve $a_0\cup a_1$ is part of a single level set $S_{t_0} \cap \Sigma$, but only $a_0$ is part of the curve $c_{t_0}$ described in the lemma, whereas the multicurve $b_0\cup b_1$ forms the multicurve $c_{t_1}$ in the lemma.}
\label{fig:box}
\end{figure}

\begin{lem} \label{lem:6.1} Let $M$ be a compact hyperbolic 3-manifold with totally geodesic boundary $ \d M = S$. Let $\e$ be smaller than the injectivity radius of $M$ and such that $N_\e(S)$ is an embedded tubular neighborhood. Let $\{S_t\}$ be the leaves of the foliation of $N_\e(S)$ by surfaces equidistant from $S$.  Let $\Sigma$ be a smooth incompressible proper not necessarily immersed surface in $M$ that is transverse to the foliation $\{S_t\}$ except at isolated points. Let $c = \d \Sigma$. By transversality, for generic $t$ the multicurve $c_t$ given by the part of $S_t\cap \Sigma$ that cobounds a subsurface of $\Sigma$ with $c = c_0$ is a smooth multicurve.  Let $T$ be the set (of full measure) of all $t\in[0,\e)$ such that $c_t$ is a smooth multicurve. Since $S$ is totally geodesic, each multicurve $c_t$ is homotopic to a possibly degenerate geodesic multicurve $\gamma_t$ in $S$. Let $\Sigma_\e$ be the part of $\Sigma$ contained in $N_\e(S)$. Then for $C = 1/\e$, we have that $$\inf\limits_{t\in T} |\gamma_t|_{S} \leq C \emph{Area}(\Sigma_\e).$$
\end{lem}

\begin{proof} The coarea formula implies the inequality $\inf\limits_{t\in T} |c_t|_{S_t} \leq C \text{Area}(\Sigma_\e)$.  Since $S$ is totally geodesic, for all $t\in T$, we have $|\gamma_t|_S\leq |c_t|_{S_t}.$  \end{proof}
Note that in the previous lemma, when $\inf\limits_{t\in T} |\gamma_t|_S$ is zero, because $\Sigma$ is incompressible and any loop with length less than $\inj(M)$ bounds a disk, every component of $\Sigma$ can either be homotoped to be disjoint from $N_\e(S)$ or be contained in $S$.

The next proposition requires a notion of geometric complexity for homology classes. For any compact Riemannian manifold $M$ one can define the stable norm on the first homology of $M$ (see \cite{Gromovmetric} Section 4C). The mass of a Lipschitz 1-chain $\alpha = \sum_i t_i\alpha_i$ in $M$ is defined to be $\text{mass}(\alpha) = \sum_i |t_i|\text{length}(\alpha_i).$ The mass of a class $a\in H_1(M)$ is then the infimal value of the mass of a chain $\alpha$ representing $a$.
For a class $a\in H_1(M)$, the stable norm of $a$ is then given by $$||a||_{s,M} = \inf\limits_{m>0}\frac{\text{mass}(m a)}{m}.$$

Stable commutator length can also be generalized to geodesic multicurves (see Section 2.6 of \cite{Calegari}), which can naturally be viewed as Lipschitz chains. Suppose $\gamma_i\in \pi_1 M$ and $ \sum_i n [\gamma_i] = 0$ in $H_1(M)$. Let $\gamma$ be the geodesic multicurve, which is not necessarily simple, consisting of the geodesic loops determined by $\gamma_i$. Say a surface $f:S\to M$ is admissible of degree $n(S)$ if it has no closed components and $\d S$ is a union of circles $S^1_i$ with $f|_{S^1_i}$ a degree $n(S)$ cover of $\gamma_i$.
Then we define stable commutator length of $\gamma$ to be $$\scl(\gamma) = \inf\limits_{S \text{ admissible}} \frac{\chi_-(S)}{2n(S)}.$$ When $\gamma$ is a single loop, this definition agrees with the usual definition of stable commutator length.

 \begin{prop} \label{prop:6.2} Let $M$ be a compact oriented hyperbolic 3-manifold with totally geodesic boundary $\d M = S$. Let $\gamma$ be a geodesic multicurve in $S$ that is rationally nullhomologous in $M$. Then there is a constant $D> 0$ depending only on $M$ such that $$||[\gamma]||_{s,S}\leq D\scl_M(\gamma).$$ \end{prop}

 \begin{proof} If $\gamma$ is nullhomologous in $S$, then the left hand side is zero and the inequality holds. Assume now that $[\gamma]\neq 0\in H_1(S)$. Fix $\delta>0$. Let $\Sigma_m$ be an incompressible admissible surface for $\gamma$ of degree $m = n(S)$ such that $\chi_-(\Sigma_m)/2m -\scl(\gamma) < \delta$. We can triangulate $\Sigma_m$ so that there is a single vertex on each boundary component. This triangulation has $4g +3b - 4$ faces, where $g$ is the genus of $\Sigma_m$ and $b$ the number of boundary components. We can then straighten this triangulation to obtain a piecewise totally geodesic triangulated surface. Replace $\Sigma_m$ with this surface. Since every face of this triangulation of $\Sigma_m$ is geodesic, every face has area at most $\pi$. Since there are $4g+3b - 4$ faces and $\chi_-(\Sigma_m) = 2g-2 + b$, we can estimate  $$\text{Area}(\Sigma_m) \leq 3\pi\chi_-(\Sigma_m).$$

We can perturb $\Sigma_m$ to obtain a smooth surface $\Sigma’_m$ that it is transverse the foliation of $N_\e(S)$ except at isolated points and in doing so increase the area by less than $\delta$. Let $\gamma_t$ be the family of multicurves in Lemma \ref{lem:6.1}  applied to $\Sigma’_m$. Since each curve $\gamma_t$ cobounds a surface in $S$ with $\d \Sigma_m$, they are homologous, thus $||[\gamma_t]||_{s,S} = m||[\gamma]||_{s,S}$. Since $||[\gamma_t]||_{s,S}\leq |\gamma_t|_S$,
Lemma \ref{lem:6.1}  implies that $$m||[\gamma]||_{s,S}\leq C\text{Area}(\Sigma_m’) \leq C \text{Area}(\Sigma_m) + C\delta \leq 3C\pi\chi_-(\Sigma_m) + C\delta.$$
From this we get  $$||[\gamma]||_{s,S}\leq 6C\pi\chi_-(\Sigma_m)/2m + C\delta/m\leq 6C\pi\scl_M(\gamma) + 6C\pi\delta +C\delta/m.$$
Since the stable commutator length of a nontrivial rational commutator is bounded away from zero by a constant only depending on $M$, by Theorem 3.9 in \cite{Calegari}, we can replace $C$ with a larger constant $D$ such that $$||[\gamma]||_{s,S}\leq D\scl_M(\gamma),$$ as desired. \end{proof}

 We now introduce the family of manifolds that we use in our construction. The family $\{W_n\}$ of manifolds we study are easily understood using the model manifold theory of \cite{BMNS}. In particular, there is a $K$-biLipschitz map between $W_n$ and a model manifold $M_n$, where $K$ is independent of $n$. The base of the construction is Thurston’s tripus manifold $W$ (see \cite{thurstonbook}, Section 3.3.12), a hyperbolic manifold with totally geodesic boundary, and a psuedoAnosov homeomorphism $f$ of the boundary surface $\d W$. The model manifold $ M_n$ is a degree $n$ cyclic cover of the mapping torus $M_{f}$ cut open along a fiber with two oppositely oriented copies of $W$, denoted $W^+$ and $W^-$ glued as described in \cite{BMNS} Section 2.15 to the two boundary components of the cut open mapping torus. This decomposes $W_n$ into three pieces, a product region $S\times [0, n]$ and the caps $W^+$ and $W^-$ in a metrically controlled way. It will be convenient to set $M^+ = W^+\subset M_n$ and $M^- = W^-\subset M_n$ when talking about the caps of the model manifold $M_n$ for fixed $n$, and to let $W^+$ and $W^-$ denote the images of these spaces under the natural inclusion into $W_n$.

 \vspace{1cm}
\begin{figure}[H]
\labellist
\small\hair 2pt
 \pinlabel {$M^+$} [ ] at 950 1100
 \pinlabel {$M^-$} [ ] at 2000 1100
 \pinlabel {$S\times[0,n]$} [ ] at 1500 1600
\endlabellist
\centering
\includegraphics[scale=.15]{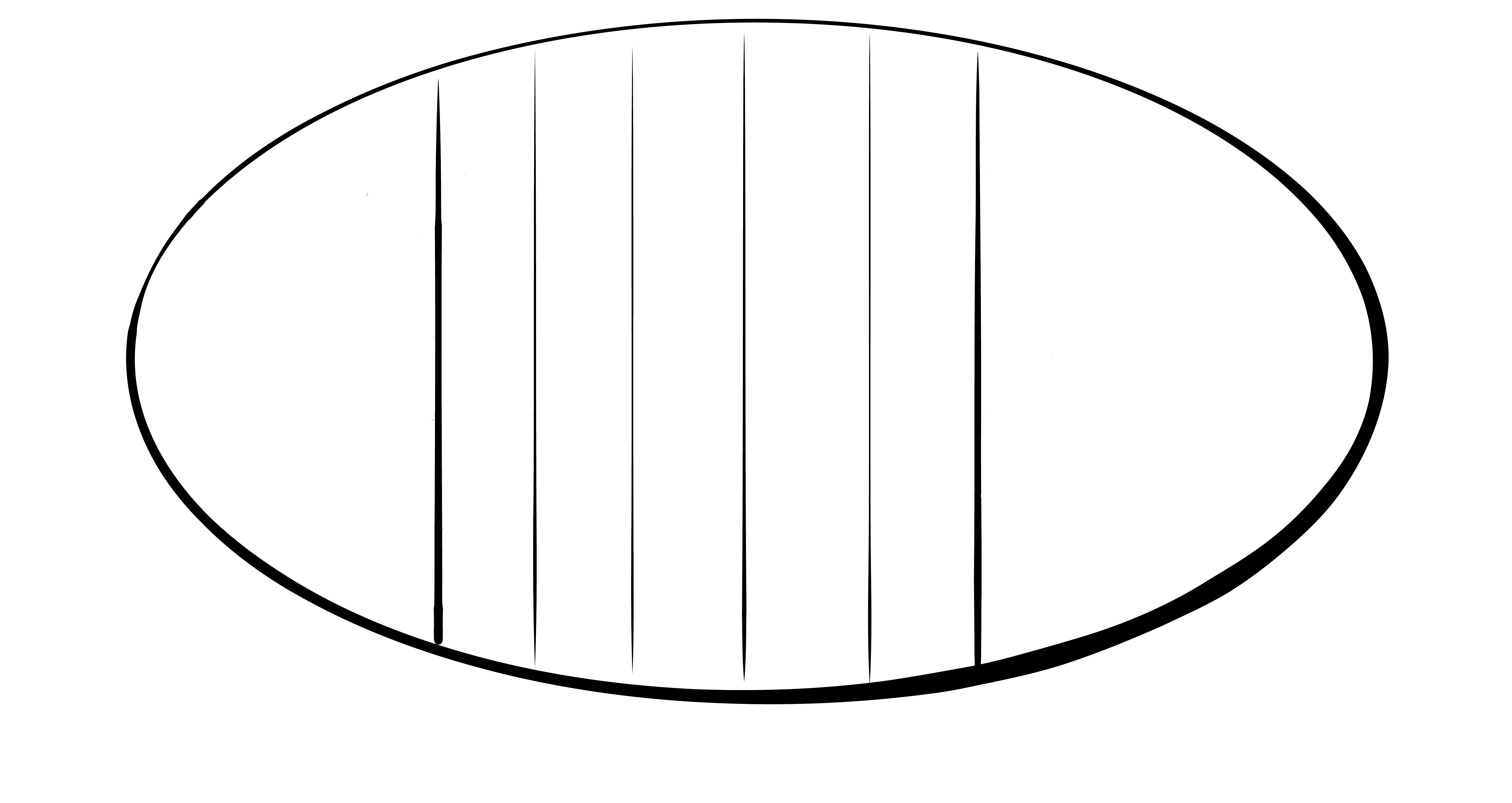}
\caption{ A schematic picture of the model manifold $M_n$ with caps $M^+$ and $M^-$ two oppositely oriented copies of the tripus manifold.}
\label{fig:caps_schematic}
\end{figure}

Given a multicurve $c$ in $M^{\pm}$, we say $c$ \textbf{bounds on both sides} if there are incompressible surfaces $S^+$ in $M^+$ and $S^-$ in $M^-$ both with boundary homotopic to $c$.

We encode the construction and its essential properties in the following proposition.

\begin{prop}  \label{prop:6.3}There is a family $\{W_n\}$ of closed hyperbolic 3-manifolds with injectivity radius uniformly bounded below and volume growing linearly in $n$ constructed from the tripus and a pseudoAnosov $f$ as described above. Each manifold $W_n$ is $K$-biLipschitz equivalent to the model manifold $M_n$ for some constant $K$ independent of $n$. Any homologically nontrivial loop in $H_1(\d W^{\pm})$ that bounds a surface in $M^{\pm}$ cannot bound on both sides. The pseudoAnosov $f$ is such that for any nonzero class $a\in H_1(\d W^+)$, the stable norm of $f_*^n(a)$ grows exponentially.
\end{prop}

\begin{proof}
Let $W$ be Thurston’s tripus manifold, a compact hyperbolic 3-manifold with totally geodesic boundary a genus 2 surface for which the inclusion map $H_1(\d W;\Z)\to H_1(W;\Z)$ is onto.
The homology of the boundary $\d W$ decomposes as the direct sum of rank 2 submodules $U$ and $V$, where $V\subset H_1(\d W)$ is the image of the boundary map $\d : H_2(W,\d W;\Z)\to H_1(\d W;\Z)$ (which is also the kernel of the inclusion $H_1(\d W)\to H_1(W)$) and $U$ is a compliment of $V$ (note that the inclusion map $H_1(\d W)\to H_1(W)$ restricted to $U$ is an isomorphism).
Let $S$ be a genus 2 surface, which we will use to mark the boundaries of $W^{+}$ and $W^-$. Assume $H_1(S;\Z)$ is generated by $e_1,~e_2,~e_3,~e_4$. Choose a marking $S\to \d W^+$ so in $W^+$ one has $U = \langle e_1,e_2\rangle $ and $V = \langle e_3, e_4\rangle$.
Similarly, choose a marking $S\to \d W^-$ so that in $W^-$ one has $V = \langle e_1,e_2\rangle $ and $U = \langle e_3, e_4\rangle$. We then define $$W_n = W^+\cup_{f^n}W^-$$ where $f:S\to S$ is a pseudo-Anosov that acts on $H_1(S)$ by the symplectic matrix\[ F = \begin{pmatrix}
 2 &  1 & 0 & 0 \\
 1 & 1 & 0 & 0 \\
 0 & 0 & 1 & -1\\
 0 & 0 & -1 & 2
\end{pmatrix} \] For the existence of such a pseudoAnosov mapping class, see the proof Lemma 7.1 in \cite{BD}. This matrix preserves the subspace decomposition above, and so ensures that every curve in $\d W^{\pm}$ that is not nullhomologous in $\d W^{\pm}$ but bounds a surface in $M^{\pm}$ cannot bound on both sides.

The mapping class $f$ acts as an Anosov matrix on $U$ and $V$. This ensures the standard Euclidean $\ell^2$-norm $||F^n(a)||_E$ of an element $a\in H_1(S)$ grows exponentially in $n$ (indeed, for our choice of $F$, it grows like $(\frac{3+\sqrt{5}}{2})^n$). Since norms on finite dimensional real vector spaces are comparable, there is a constant comparing the stable norm induced by the metric inherited from $W$ to the standard Euclidean $\ell^2$-norm on $H_1(S)$.

Lemma 7.3 in \cite{BD} explains how Theorem 8.1 in \cite{BMNS} implies that for large $n$ the manifolds $W_n$ admit a $K$-biLipschitz diffeomorphism $\mu$ from the model manifold $ M_n$ as described above. After increasing $K$, we can drop the large $n$ condition. This then also implies the linear volume growth and injectivity radius bounds.
\end{proof}

\begin{remark}
Using the model manifold, one can easily estimate the Cheeger constant of $W_n$, which will decay like $1/n$.
\end{remark}

\begin{mainthm}  \label{thm:C} The family $W_n$ of closed hyperbolic 3-manifolds from Proposition \ref{prop:6.3}  has 1-form Laplacian spectral gap that vanishes exponentially fast in relation to volume:
$$\sqrt{\lambda(W_n)}\leq B\vol(W_n)e^{-r\vol(W_n)},$$
where $r$ and $B$ are positive positive constants and $\lambda(W_n)$ is the first positive eigenvalue of the 1-form Laplacian on $W_n$.
\end{mainthm}

\begin{proof}
We continue using notation introduced in the previous propositions. Take $\gamma$ in $\d M^+\subset M_n$ to be an embedded geodesic loop representing the class $e_1 \in U \subset H_1(\d W^+)$. Recall from the proof of Proposition \ref{prop:6.3} that $\gamma\subset \d M^+$ does not bound a surface in $M^+$ but that $f^n(\gamma)\subset \d M^-$ bounds a surface in $M^-$. Let $\alpha_n = f^n(\gamma)\subset \d M^- \subset M_n$. Note that $\alpha_n$ and $\gamma$ are isotopic in $M_n$.

Fix $\delta>0$. Consider some positive integer $m$ and  incompressible surface $\Sigma_m$ bounding $\alpha_n^m$ in $M_n$ with $\chi_-(\Sigma_m)/2m - \scl_{M_n}(\gamma) < \delta$ and which minimizes $\chi_-$ among surfaces with boundary $\alpha_n^m$. We can then replace $\Sigma_m$ with a homotopic surface that pushes the boundary of $\Sigma_m$ into the interior of $M^-$ and which intersects $\d M^-$ transversely and essentially in both $\d M^-$ and $\Sigma_m$.
We can then attach an annulus to $\Sigma_m$ cobounding $\alpha_n^m$ and the boundary of the modified surface $\Sigma_m$. This new $\Sigma_m$ bounds $\alpha_n^m$ with a collar neighborhood of the boundary contained entirely in $M^-$ and intersects $\d M^-$ transversely in a union of loops essential in both $\Sigma_m$ and $\d M^-$.

We focus on the portion of $\Sigma_m$ that lies in $M^-$. Define $\Sigma^-_m =\Sigma_m\cap M^-$. If $\Sigma_m$ is contained in $M^-$, then Proposition \ref{prop:6.2}  applied to $\alpha_n^m$ in $M^-$ implies that
$$||[\alpha_n^m]||_{s,\d M^-} = m||[\alpha_n]||_{s,\d M^-} \leq m\scl_{M^-}(\alpha_n)\leq D\chi_-(\Sigma_m^-) = D\chi_-(\Sigma_m) ,$$
where $||\cdot||_{s, \d M^-}$ is the stable norm of $H_1(\d M^-)$. Since $\chi_-(\Sigma_m)/m- \scl_{M_n}(\gamma)\leq \delta$, we conclude $$||[\alpha_n]||_{s,\d M^-}\leq D\scl_{M_n}(\gamma) + D\delta.$$

Our goal now is to get this same estimate for the other possible ways $\Sigma_m$ sits in $M_n$.
\vspace{1cm}
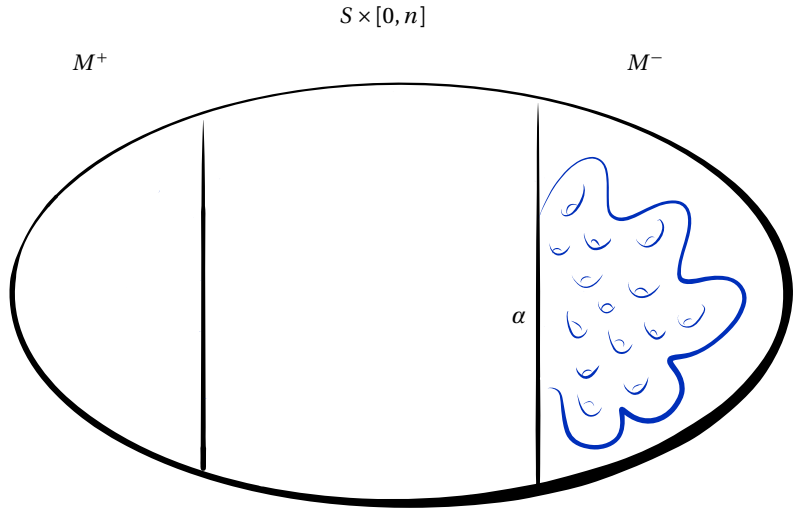
\begin{figure}[H]
\labellist
\small\hair 2pt
 \pinlabel {$M^+$} [ ] at 900 1560
 \pinlabel {$S\times[0,n]$} [ ] at 1480 1650
 \pinlabel {$M^-$} [ ] at 2000 1560
 \pinlabel {$\alpha$} [ ] at 1750 1050
\endlabellist
\centering
\includegraphics[scale=.2]{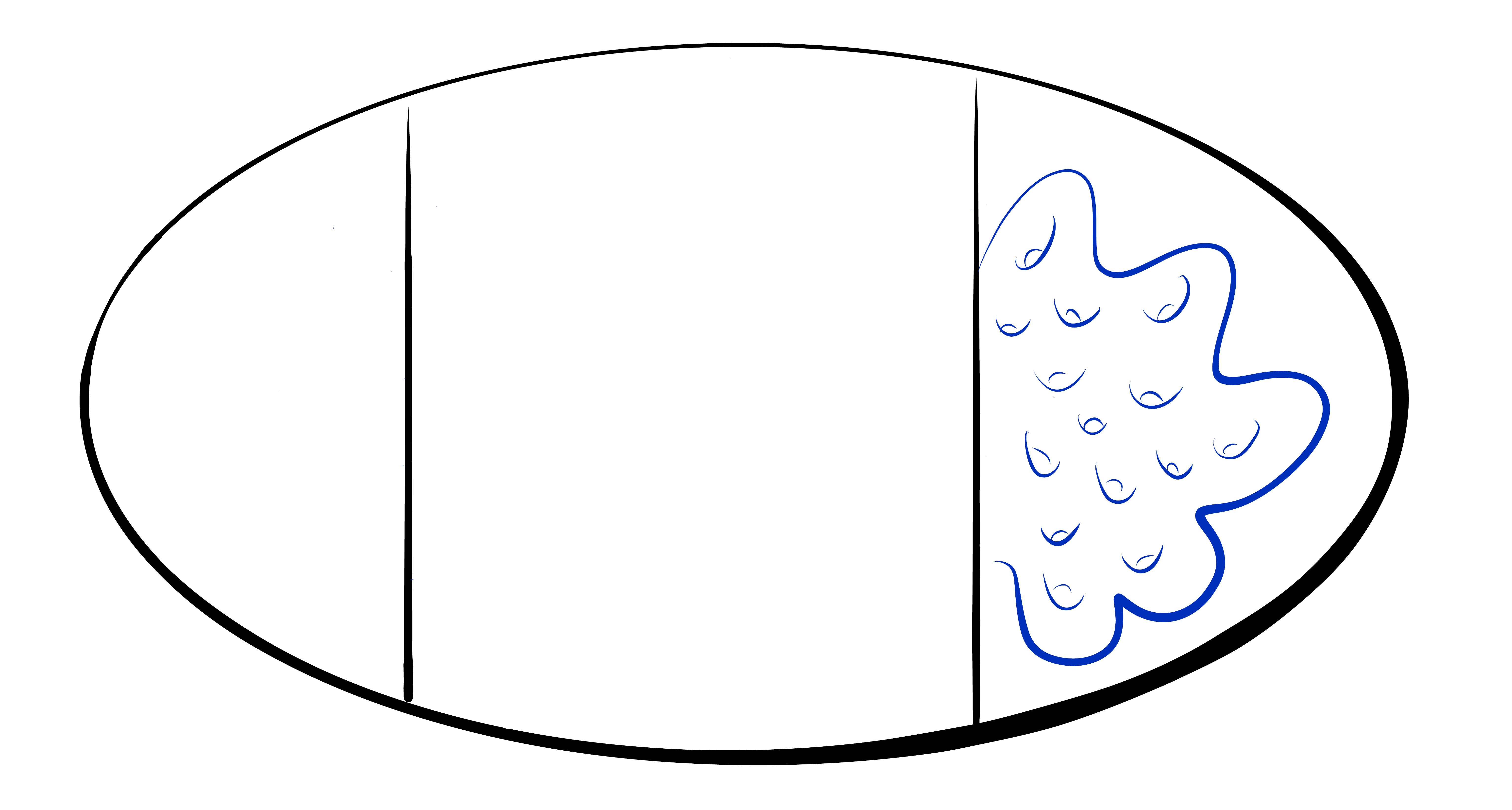}
\caption{ A schematic picture of the simplest case of a surface bounding $\alpha$ in $M^-$.}
\label{fig:alpha_bounds}
\end{figure}

Consider the case that $\Sigma_m$ does not lie entirely in $M^-$. There are two possibilities. The first involves the surface $\Sigma_m$ passing into the product region but not intersecting $M^+$. In this case the surface can be homotoped to lie in $M^-$, so that Proposition \ref{prop:6.2} applies, giving the desired estimate as in the previous case.
\vspace{2cm}
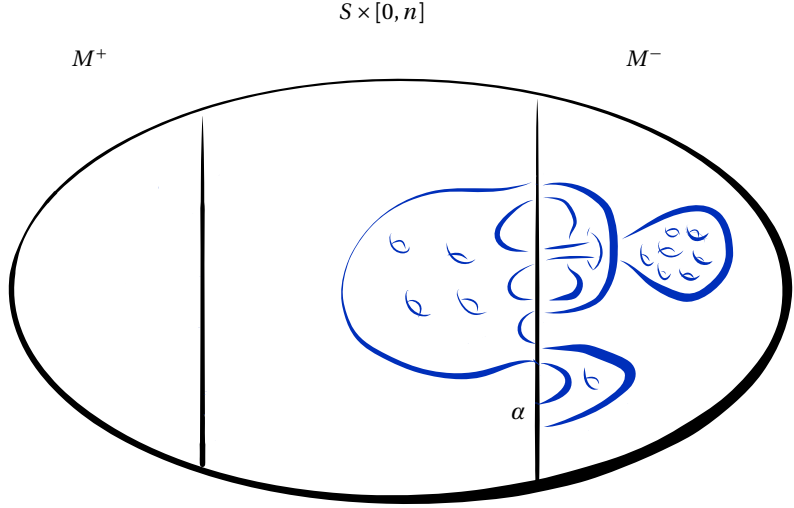
\begin{figure}[H]
\labellist
\small\hair 2pt
 \pinlabel {$M^+$} [ ] at 900 1560
 \pinlabel {$S\times[0,n]$} [ ] at 1480 1650
 \pinlabel {$M^-$} [ ] at 2000 1560
 \pinlabel {$\alpha$} [ ] at 1750 850
\endlabellist
\centering

\includegraphics[scale=.2]{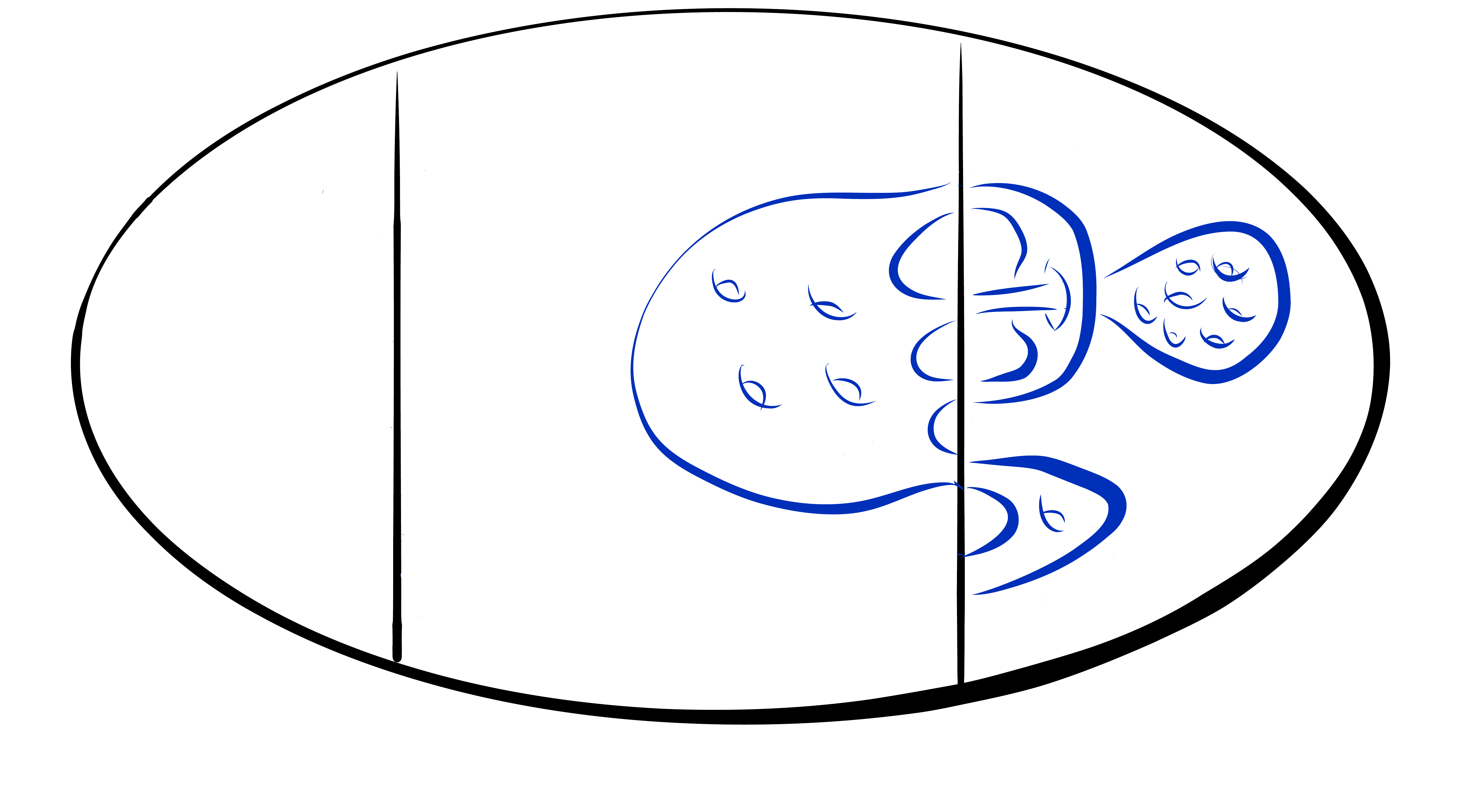}
\caption{ A schematic picture of a surface bounding $\alpha$ that passes back into $M^-$ but does not pass into $M^+$.}
\label{fig:pass_back_schematic}
\end{figure}

\vspace{2cm}
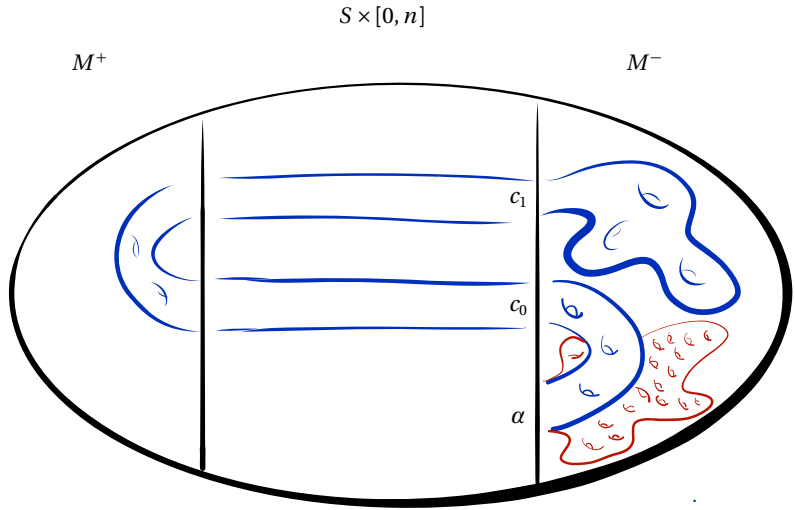
\begin{figure}[H]
\labellist
\small\hair 2pt
 \pinlabel {$M^+$} [ ] at 900 1560
 \pinlabel {$S\times[0,n]$} [ ] at 1480 1650
 \pinlabel {$M^-$} [ ] at 2000 1560
 \pinlabel {$\alpha$} [ ] at 1750 850
 \pinlabel {$c_1$} [ ] at 1750 1285
 \pinlabel {$c_0$} [ ] at 1750 1070

\endlabellist
\centering

\includegraphics[scale=.2]{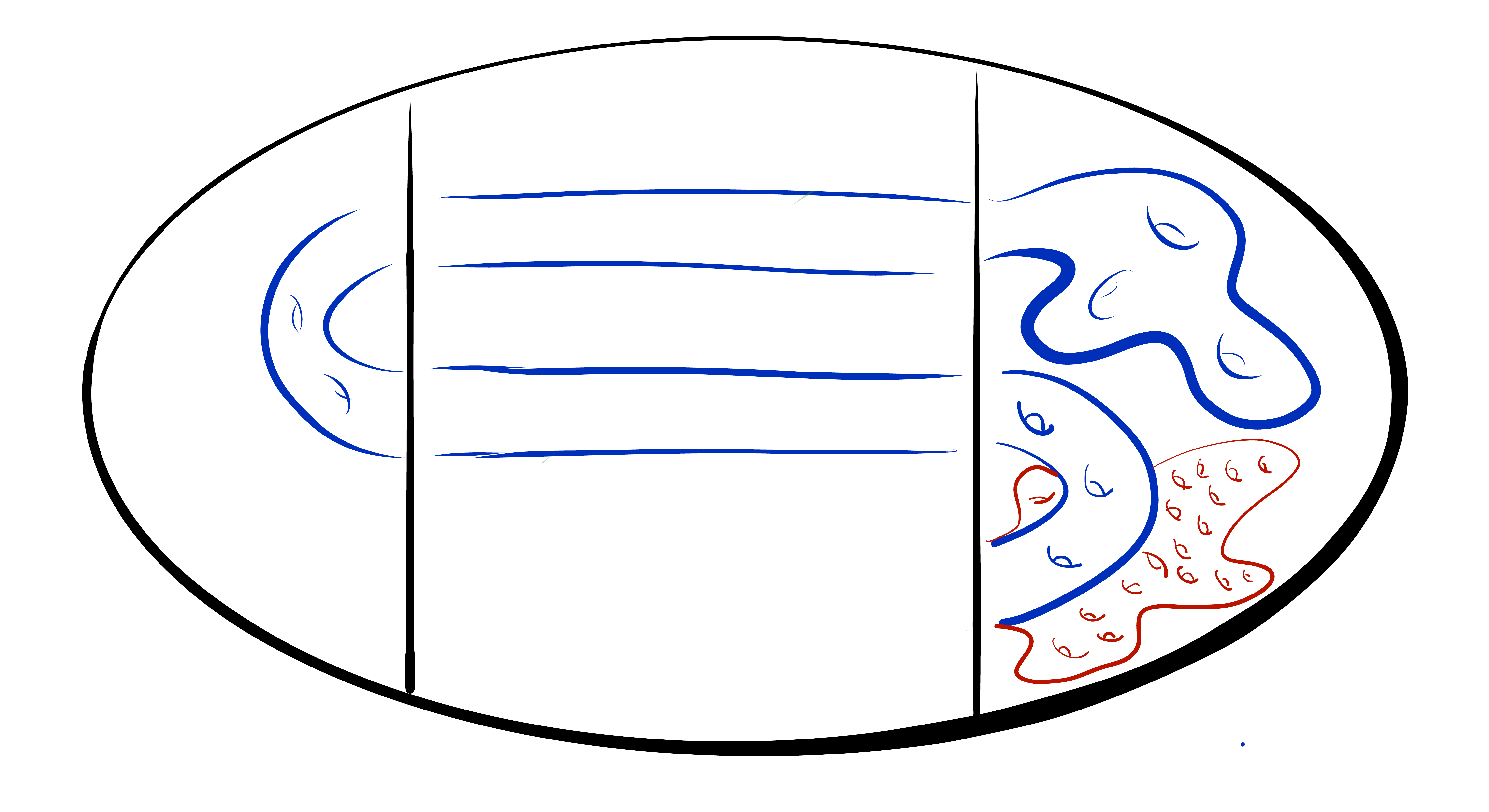}
\caption{ A schematic picture of a surface bounding $\alpha$ that passes back into $M^+$. Notice the multicurve $c^- = c_0\cup c_1$ bounds surfaces in $M^+$ and $M^-$, so is homologically trivial. }
\label{fig:bounds_both_sides}
\end{figure}

The second possibility concerns the surface $\Sigma_m$ crossing through the product region into $M^+$ with an essential intersection with $\d \Sigma^+$. In this case, we will see that the surface $\Sigma_m^-$ has boundary homologous to $\alpha_n^m$, which will allow us to apply Proposition \ref{prop:6.2} to obtain the desired estimate. By construction, a sufficiently small collar $C$ of the boundary $\d \Sigma_m$ in $\Sigma_m$ maps into $M^-$, so in particular, a subsurface of $\Sigma_m^-$ has some boundary component that maps to $\alpha_n^m$. That boundary component can be closed by attaching a surface $S^-$ that bounds $\alpha_n^m$ in $M^-$ to $\Sigma_m$.
From this, we see that the multicurve $c^- = \d \Sigma^-_m -\alpha_n $ bounds surfaces in $M^+$ and $M^-$.
Thus by Proposition \ref{prop:6.3}, $c^-$ must be homologically trivial in $\d M^{-}$. Let $x = \d\Sigma_m^- = c^- + \alpha_n^m$. Since $c^-$ is nullhomologous, $||[x]||_{s,\d M^-} = ||[\alpha_n^m]||_{s,\d M^-} = m||[\alpha_n]||_{s,\d M^-}.$
By Proposition \ref{prop:6.2} , $||[x]||_{s,\d M^-}\leq D\scl_{M^-}(x).$ Since $x$ is essential in $\Sigma_m$, we get that $\chi_-(\Sigma_m^-) \leq \chi_-(\Sigma_m)$, then using that $\chi_-(\Sigma_m)/2m - \delta \leq \scl_{M_n}(\alpha_n)$,
we obtain $\scl_{M^-}(x) \leq \chi_-(\Sigma_m^-)/2 \leq m\scl_{M_n}(\alpha_n) + \delta m$.
Putting this all together and dividing by $m$, we get that $$||[\alpha_n]||_{s,\d M^-}\leq D\scl_{M_n}(\alpha_n) + D\delta.$$

We therefore have in each case that there is a constant $D$ independent of $n$ such that $$||[\alpha_n]||_{s,\d M^-}\leq D\scl_{M_n}(\alpha_n) + D\delta.$$ By Proposition \ref{prop:6.3} , $||[\alpha_n]||_{s,\d M^-} = ||[f^n(\gamma)]||_{s,\d M^+}$ grows exponentially in $n$.
Thus for some constants $B>0$ and $r >0$, we have $$Be^{rn}\leq D\scl_{M_n}(\gamma) + D\delta,$$ where we use that $\gamma$ and $\alpha_n$ are homotopic in $M_n$. Using the injectivity radius lower bound and Theorem 3.9 of \cite{Calegari}, we can increase $D$ and drop the additive constant in this inequality. By Proposition \ref{prop:6.3} , the volume growth of the $W_n$ is proportional to $n$ so there is a constant $C$ such that $\vol(W_n)\leq Cn$.
Additionally, using the $K$-biLipschitz comparison of Proposition \ref{prop:6.3} , the length of $\gamma$ in $W_n$ is bounded from above by $2K|\gamma|_{W}$, where $W$ is the tripus. As a result, Theorem A implies that the spectral gap for the 1-form Laplacian of the manifolds $W_n$ vanishes exponentially fast in $n$.
In particular, we have \[\sqrt{\lambda(W_n)} \leq A\vol(W_n)\frac{|\gamma|_{W_n}}{\scl_{W_n}(\gamma)} \leq 2KACB^{-1}D|\gamma|_Wne^{-rn},\] so the result holds after redefining $B$ to be $2KACB^{-1}D|\gamma|_W.$

\end{proof}

%
\bibliographystyle{alpha}

\bibliography{stableisohodge}

\newcommand{\etalchar}[1]{$^{#1}$}
\begin{thebibliography}{ABB{\etalchar{+}}17}

\bibitem[ABB{\etalchar{+}}17]{samurai}
Miklos Abert, Nicolas Bergeron, Ian Biringer, Tsachik Gelander, Nikolay
  Nikolov, Jean Raimbault, and Iddo Samet.
\newblock On the growth of {$L^2$}-invariants for sequences of lattices in
  {L}ie groups.
\newblock {\em Ann. of Math. (2)}, 185(3):711--790, 2017.

\bibitem[Bav91]{Bavard}
Christophe Bavard.
\newblock Longueur stable des commutateurs.
\newblock {\em Enseign. Math. (2)}, 37(1-2):109--150, 1991.

\bibitem[BD15]{BDE}
Jeffrey~F. Brock and Nathan~M. Dunfield.
\newblock Injectivity radii of hyperbolic integer homology 3-spheres.
\newblock {\em Geom. Topol.}, 19(1):497--523, 2015.

\bibitem[BD17]{BD}
Jeffrey~F. Brock and Nathan~M. Dunfield.
\newblock Norms on the cohomology of hyperbolic 3-manifolds.
\newblock {\em Invent. Math.}, 210(2):531--558, 2017.

\bibitem[BDG18]{Bois}
Jean-Daniel Boissonnat, Ramsay Dyer, and Arijit Ghosh.
\newblock Delaunay triangulation of manifolds.
\newblock {\em Found. Comput. Math.}, 18(2):399--431, 2018.

\bibitem[Bel14]{minvol2}
Mikhail Belolipetsky.
\newblock Hyperbolic orbifolds of small volume.
\newblock In {\em Proceedings of the {I}nternational {C}ongress of
  {M}athematicians---{S}eoul 2014. {V}ol. {II}}, pages 837--851. Kyung Moon Sa,
  Seoul, 2014.

\bibitem[BMNS16]{BMNS}
Jeffrey Brock, Yair Minsky, Hossein Namazi, and Juan Souto.
\newblock Bounded combinatorics and uniform models for hyperbolic 3-manifolds.
\newblock {\em J. Topol.}, 9(2):451--501, 2016.

\bibitem[Bre97]{bredon}
Glen~E. Bredon.
\newblock {\em Topology and geometry}, volume 139 of {\em Graduate Texts in
  Mathematics}.
\newblock Springer-Verlag, New York, 1997.
\newblock Corrected third printing of the 1993 original.

\bibitem[B{\c{S}}V16]{BSV}
Nicolas Bergeron, Mehmet~Haluk {\c{S}}eng\"{u}n, and Akshay Venkatesh.
\newblock Torsion homology growth and cycle complexity of arithmetic manifolds.
\newblock {\em Duke Math. J.}, 165(9):1629--1693, 2016.

\bibitem[BV13]{BV}
Nicolas Bergeron and Akshay Venkatesh.
\newblock The asymptotic growth of torsion homology for arithmetic groups.
\newblock {\em J. Inst. Math. Jussieu}, 12(2):391--447, 2013.

\bibitem[Cal08]{length}
Danny Calegari.
\newblock Length and stable length.
\newblock {\em Geom. Funct. Anal.}, 18(1):50--76, 2008.

\bibitem[Cal09]{Calegari}
Danny Calegari.
\newblock {\em scl}, volume~20 of {\em MSJ Memoirs}.
\newblock Mathematical Society of Japan, Tokyo, 2009.

\bibitem[Can75]{cantor}
M.~Cantor.
\newblock Sobolev inequalities for {R}iemannian bundles.
\newblock In {\em Differential geometry ({P}roc. {S}ympos. {P}ure {M}ath.,
  {V}ol. {XXVII}, {S}tanford {U}niv., {S}tanford, {C}alif., 1973), {P}art 2},
  pages 171--184, 1975.

\bibitem[Dod76]{Dodziuk}
Jozef Dodziuk.
\newblock Finite-difference approach to the {H}odge theory of harmonic forms.
\newblock {\em Amer. J. Math.}, 98(1):79--104, 1976.

\bibitem[Dod81]{dodzuik2}
Jozef Dodziuk.
\newblock Sobolev spaces of differential forms and de {R}ham-{H}odge
  isomorphism.
\newblock {\em J. Differential Geometry}, 16(1):63--73, 1981.

\bibitem[DP76]{dP}
J.~Dodziuk and V.~K. Patodi.
\newblock Riemannian structures and triangulations of manifolds.
\newblock {\em J. Indian Math. Soc. (N.S.)}, 40(1-4):1--52 (1977), 1976.

\bibitem[GMM09]{minvol}
David Gabai, Robert Meyerhoff, and Peter Milley.
\newblock Minimum volume cusped hyperbolic three-manifolds.
\newblock {\em J. Amer. Math. Soc.}, 22(4):1157--1215, 2009.

\bibitem[Gro07]{Gromovmetric}
Misha Gromov.
\newblock {\em Metric structures for {R}iemannian and non-{R}iemannian spaces}.
\newblock Modern Birkh\"{a}user Classics. Birkh\"{a}user Boston, Inc., Boston,
  MA, english edition, 2007.
\newblock Based on the 1981 French original, With appendices by M. Katz, P.
  Pansu and S. Semmes, Translated from the French by Sean Michael Bates.

\bibitem[L{\^{e}}18]{thang}
Thang T.~Q. L{\^{e}}.
\newblock Growth of homology torsion in finite coverings and hyperbolic volume.
\newblock {\em Ann. Inst. Fourier (Grenoble)}, 68(2):611--645, 2018.

\bibitem[LL21]{LL}
Francesco Lin and Michael Lipnowski.
\newblock The seiberg-witten equations and the length spectrum of hyperbolic
  three-manifolds.
\newblock {\em J. Amer. Math. Soc.}, 2021.
\newblock published electronically (to appear in print).

\bibitem[LS18]{LS}
Michael Lipnowski and Mark Stern.
\newblock Geometry of the smallest 1-form {L}aplacian eigenvalue on hyperbolic
  manifolds.
\newblock {\em Geom. Funct. Anal.}, 28(6):1717--1755, 2018.

\bibitem[L{\"{u}}c16]{wolf}
Wolfgang L{\"{u}}ck.
\newblock Approximating {$L^2$}-invariants by their classical counterparts.
\newblock {\em EMS Surv. Math. Sci.}, 3(2):269--344, 2016.

\bibitem[Sab17]{Sab}
St\'{e}phane Sabourau.
\newblock Small volume of balls, large volume entropy and the {M}argulis
  constant.
\newblock {\em Math. Ann.}, 369(3-4):1557--1571, 2017.

\bibitem[Sch95]{schwarz}
G\"{u}nter Schwarz.
\newblock {\em Hodge decomposition---a method for solving boundary value
  problems}, volume 1607 of {\em Lecture Notes in Mathematics}.
\newblock Springer-Verlag, Berlin, 1995.

\bibitem[Ste13]{stern}
Ari Stern.
\newblock {$L^p$} change of variables inequalities on manifolds.
\newblock {\em Math. Inequal. Appl.}, 16(1):55--67, 2013.

\bibitem[Thu97]{thurstonbook}
William~P. Thurston.
\newblock {\em Three-dimensional geometry and topology. {V}ol. 1}, volume~35 of
  {\em Princeton Mathematical Series}.
\newblock Princeton University Press, Princeton, NJ, 1997.
\newblock Edited by Silvio Levy.

\end{thebibliography}

\end{document}